\documentclass[12pt]{amsart}
\usepackage{
  fullpage,
  graphicx,amssymb,epic,eepic,cancel,tikz,
color,hyphenat,
  multicol, todonotes,bbold, epstopdf,
}
\usepackage[all]{xy}

\usepackage{tkz-graph}
\GraphInit[vstyle = Shade]
\tikzset{
  LabelStyle/.style = { rectangle, rounded corners, draw,
                        minimum width = 1em, fill = yellow!50,
                        text = red, font = \bfseries },
  VertexStyle/.append style = { inner sep=5pt,
                                font = \Large\bfseries},
  EdgeStyle/.append style = { bend left=10} }
\usetikzlibrary{arrows,chains,matrix,positioning,scopes}
\makeatletter
\tikzset{join/.code=\tikzset{after node path={%
\ifx\tikzchainprevious\pgfutil@empty\else(\tikzchainprevious)%
edge[every join]#1(\tikzchaincurrent)\fi}}}
\makeatother
\tikzset{>=stealth',every on chain/.append style={join},
         every join/.style={->}}
\tikzstyle{labeled}=[execute at begin node=$\scriptstyle,
   execute at end node=$]
\newtheorem{theorem}{Theorem}[section]
\newtheorem{lemma}[theorem]{Lemma}
\newtheorem{problem}[theorem]{Problem}
\newtheorem{proposition}[theorem]{Proposition}
\newtheorem{corollary}[theorem]{Corollary}

\newtheorem{definition}{Definition}
\newtheorem{remark}{Remark}

\def\Z{{\mathbb Z}}
\def\C{{\mathbb C}}
\def\k{{\Bbb k}}
\def\A{\widehat{A}}
\def\la{\lambda}

\def\cal{\mathcal}

\def\Hom{\operatorname{Hom}}

\def\la{\langle}
\def\ra{\rangle}

\def\b{\beta}

\def\P{\mathfrak{P}}
\def\B{\mathcal{B}}

\def\s{\sigma}
\def\G{\Gamma}

\def\D{\mathcal{D}}

\def\Be{\mathcal Be}
\def\Sig{\Sigma}

\title{On the 2-linearity of the free group}

\author{Anthony M. Licata}
\address{Mathematical Sciences Institute; Australian National University; Canberra, Australia}
\email{amlicata@gmail.com}

\begin{document}
\maketitle
\section{Introduction}
Let $\G$ be a group, and $\D$ a category.  A weak action of $\G$ on $\D$ 
is a map
$$
	\psi: \G \longrightarrow \mbox{Aut}(\D), \ \ g \mapsto \psi_g,
$$
from $\G$ to the auto-equivalences of $\D$, such that
for all $h,g\in \G$,
$
	\psi_g\psi_h \cong \psi_{gh}.
$
In a genuine action of $\G$, one also specifies isomorphisms $\alpha_{g,h}: \psi_g\psi_h \rightarrow \psi_{gh}$ and requires them to satisfy an associativity constraint.  Though the entirety of this paper is concerned only with a single weak action, the one we consider can be made genuine, and it is worth noting from the outset the importance of genuine actions.  If $\k$ is a field and the category $\D$ is enriched in finite dimensional $\k$-vector spaces, we say that $\psi$ is a {\it 2-representation} of $\G$. 

Let $[\mbox{Aut}(\D)]$ denote the group of isomorphism classes of auto-equivalences of $\D$.
Given a weak action of $\G$ on $\D$, the induced map
$
	[\psi]: \G\longrightarrow [\mbox{Aut}(\D)]
$
is a group homomorphism.  It often happens in examples that the category $\D$ is also triangulated, with $\G$ acting by triangulated auto-equivalences.  We can then restrict even more and ask that the triangulated category $\D$ be finitely-generated, which means that there are finitely many objects of $\D$ with the property that the smallest full triangulated subcategory containing all of them is $\D$ itself.   We then arrive at the following definition.
\begin{definition}
A group $\G$ is {\bf 2-linear} if $\G$ admits a 2-representation $\psi$ on a finitely-generated $\k$-linear triangulated category such that the associated group homomorphism $[\psi]$ is injective.  
\end{definition} 
This notion of a 2-linear group is analogous to the notion of a linear group, which is one that embeds into the group of invertible endomorphisms of a finite-dimensional $\k$-vector space.  In the last twenty years a number of explicit examples of 2-representations of groups on triangulated categories have appeared in representation theory, symplectic and algebraic geometry, and mathematical physics.  Often in these examples, however, it is the triangulated category that is of interest, with the group and its 2-representation appearing mainly as organizing objects.  Moreover, despite the existence of interesting examples, the class of 2-linear groups remains somewhat understudied as a chapter of combinatorial and geometric group theory.  For example, almost nothing is known at present about the group theoretic consequences of 2-linearity, and it is not at all clear how different the classes of 2-linear and linear groups are from one another.  Nevertheless, it appears that the class of 2-linear groups is very rich.  In particular, the following groups are either known or conjectured to be 2-linear.
\begin{itemize}

\item {\it Coxeter groups and their braid groups.}  Coxeter groups themselves are 2-linear
\footnote{The 2-linearity of a Coxeter group $W$ of simply-laced type follows, for example, from work of Frenkel-Khovanov-Schiffmann \cite{FKS}.  The underlying triangulated category in the FKS 2-representation is a category of matrix factorizations over a non-commutative zigzag algebra.  The Grothendieck group of this 2-representation is the reflection representation of $W$, which is known to be faithful.  One can give a similar construction for other Coxeter groups using matrix factorizations of Soergel bimodules.}
while their groups are known to be 2-linear in some important special cases \cite{BravThomas,IKU,Jensen,KhovanovSeidel,Licata,RZ}. In \cite{Rouquier}, Rouquier conjectures that the 2-representation of an arbitrary Artin-Tits braid group on the homotopy category of Soergel bimodules is faithful.  

\item {\it Mapping class groups of surfaces.}  The theory of bordered Floer-homology of Lipshitz-Oszvath-Thurston provides faithful 2-representations of mapping class groups of surfaces with boundary \cite{LOT}.  As far as we are aware, it is an open problem to show that the mapping class groups of closed surfaces are 2-linear.  The question of group-theoretic consequences of 2-linearity was first posed by Lipshitz-Oszvath-Thurston.  

\item {\it Fundamental groups of complexified hyperplane arrangements.}  Many of the organizing functors (twisting, shuffling, wall-crossing, cross-walling,...) of geometric representation theory and symplectic-algebraic geometry produce 2-representations of these fundamental groups.  Of particular interest here are the 2-representations of fundamental groups on the derived category of category $\mathcal{O}$ for a symplectic resolution \cite{BezLos,BLPW,BPW,Losev}.  It is an open question to determine when these 2-representations are faithful.
\end{itemize}

The $n$-strand Artin braid group $B_n$ is perhaps the best studied example from the point of view of 2-linearity, and $B_n$ appears as an entry in each of the three classes described above.  The existing literature on 2-representations of $B_n$ also motivates the interest in studying explicit faithful 2-representations in detail.  For the first non-trivial example of a representation of $B_n$ on a vector space -- the Burau representation -- is not faithful except for small $n$.  Moreover, although the group $B_n$ is known to be linear by work of Bigelow and Krammer \cite{Bigelow, Krammer}, studying $B_n$ via the Bigelow-Krammer representation is difficult.  On the other hand, the first non-trivial example of a 2-representation of $B_n$ -- the categorification of the Burau representation -- is already faithful, as was shown first by Khovanov-Seidel \cite{KhovanovSeidel} and, for $n=3$, Rouquier-Zimmermann \cite{RZ}.  Moreover, much of the rich mathematics related to the braid group is readily visible from the lens of this first faithful 2-representation.  (See, for example, \cite{KhovanovSeidel,Thomas,RZ,BravThomas,BaldwinGrigsby}, each of which contains a explanation of how some structure of interest in the study of the braid group appears in the categorified Burau representation.)

The goal of the current paper is to define and study what we feel is a basic and prototypical example of a faithful 2-representation of a group.  We construct a faithful 2-representation of the free group $F_n$ and explore how some of the structure of $F_n$ appears from the homological algebra of this 2-representation.  Our underlying category $\D$ is the homotopy category of projective modules over a finite dimensional $\k$-algebra $\A$, where $\k$ is a field of characteristic 0.  This algebra, which is an example of a {\it zigzag algebra}, is the Koszul dual of the preprojective algebra of the doubled complete graph (see Section \ref{sec:algebras}.)
Our main result, which establishes the 2-linearity of $F_n$, is the following.

\begin{theorem}\label{thm:main}
The action of $F_n$ on $\D$ is faithful.
\end{theorem}

The zigzag algebra $\A$ has $n$ indecomposable projective modules $P_1,\dots,P_n$, and the Seidel-Thomas spherical twists in these objects define the action of the standard generators of $\s_1,\dots,\s_n$ of $F_n$.  If spherical objects $E_1,E_2$ have no morphisms between them, then the associated spherical twists 
$\Phi_{E_1}$ and $\Phi_{E_2}$ will commute; when there is a one-dimensional space of morphisms between them, on the other hand, the spherical twists will braid.  Thus, if the spherical twists in objects $P_1,\dots,P_n$ are to generate a free action of $F_n$, the dimensions of morphism spaces between each pair of the generating spherical objects must be at least 2.  Our algebra $\A$ is constructed so that $\mbox{dim} \Hom(P_i,P_j)=2$ for all $i, j=1,\dots,n$.  From this point of view, our main result Theorem \ref{thm:main} below is somewhat analogous to a theorem of Keating \cite{Keating}, who proves that if $L$ and $L'$ are two non-quasi-isomorphic Lagrangian spheres in a symplectic manifold $M$ with $\mbox{dim} \Hom(L,L') = 2$ in the Fukaya category of $M$, then the Dehn twists in $L$ and $L'$ generate a free froup of rank $2$ inside the symplectomorphism group of $M$.

Our proof of Theorem \ref{thm:main} is by ping pong.  After establishing Theorem \ref{thm:main}, we devote the remaining sections to describing how other structure of interest in the mathematics of the free group appears in the representation theory of the zigzag algebra $\A$. For instance, we describe is how two different monoids in $F_n$ arise as stabilizers of 
``non-negative" subcategories of the category $\D$.  These non-negative subcategories, in turn, arise from non-negative gradings on the zigzag algebra.  For example, one grading on $\A$ -- the $\widetilde{o}$-grading -- gives rise to a subcategory $\D_{\geq 0}^{\widetilde{o}} \subset \D$ with the property that
$g\in F_n$ preserves $\D_{\geq 0}^{\widetilde{o}}$ if and only if $g$ is in the standard positive monoid generated by $\s_1,\dots,\s_n$.  This particular grading is used crucially in our ping pong proof of Theorem \ref{thm:main}.  Another grading on $\A$
-- the $\vec{o}$-grading -- gives rise to a different subcategory $\D_{\geq 0}^{\vec{o}}$ with the property that
$g(\D_{\geq 0}^{\vec{o}})\subset \D_{\geq 0}^{\vec{o}}$ if and only if $g$ is in the dual positive monoid $F_n^+$ of $F_n$ defined by Bessis in \cite{Bessis}. 
In fact, the ping pong construction used in the proof Theorem \ref{thm:main}, which uses the baric structure determined by the $\widetilde{o}$-grading, has a dual analog which uses the $\vec{o}$-grading and the Bessis monoid.  The appearance of the Bessis monoid in relation to the $\vec{o}$-grading hints at a close relationship between the $\vec{o}$-graded representation theory of $\A$ and the topology of the punctured disk (whose fundamental group is the free group).  This relationship is partially explored in Section \ref{sec:dualpingpong} (see, for example, Proposition \ref{prop:equiv}).

The categories $\D_{\geq 0}(\widetilde{o})$ and $\D_{\geq 0}(\vec{o})$ are both the non-negative parts of {\it baric structures} in the sense of Achar-Treumann \cite{AcharTreumann}.  At the end of the paper we use the compatability of monoids and baric structures to give homological interpretations of the standard and Bessis dual word-length metrics on $F_n$ in Theorems \ref{thm:metric1} and \ref{thm:metric2}.  These theorems equate the relevant word length of $g\in F_n$ with the associated ``baric spread" of $g$ in its action on $\D$.

A third grading on $\A$, the Koszul path-length grading, gives rise to a $t$-structure $(\D^{\geq 0},\D^{\leq 0})$.  
Just as for the baric subcategory $\D_{\geq 0}(\widetilde{o})$, a group element $g\in F_n$ preserves the subcategory $\D^{\geq 0}$ if and only if $g$ is in the standard positive monoid.  However, the metric on the free group determined by this $t$ structure is different than the metrics defined by baric structures, and the little bit we know about it is mentioned at the end of the paper.  It would be interesting to give a combinatorial description, without direct reference to the category $\D$, of this exotic metric.  

The construction of the zigzag algebra $\A$ and the action of $F_n$ on the homotopy category of projective $\A$ modules is very similar to a number of  earlier constructions of braid and Artin-Tits group actions on representation categories of quivers \cite{HuerfanoKhovanov,KhovanovSeidel,Rouquier,RZ}.  Analogs of some aspects of this paper should exist for arbitrary Artin-Tits braid groups, and we plan to write about special cases of spherical and right-angled Artin-Tits groups in \cite{LicataQueffelec} and \cite{Licata}, respectively.  In the remainder of the introduction we give a more detailed explanation of the contents of the paper.

\subsection{Contents}

In Section \ref{sec:free} we collect some information about the free group to be used later, including the definition of reflections in $F_n$, the Bessis monoid $F_n^+$, and the Hurwitz action of $B_n$ on reduced expressions for the Coxeter element $\gamma = \s_1\dots \s_n$.  

In Section \ref{sec:ping}, we recall Klein's ping pong lemma and state a dual ping pong lemma using the Bessis monoid $F_n^+$.  The point of both of these lemmas is to formulate conditions which guarantee that an action of $F_n$ on a set is free.  A notable point about the dual ping pong lemma is that it only requires studying the positive Bessis monoid $F_n^+$ and not the entire free group.  Our formulation of dual ping pong using $F_n^+$ is motivated by a similar formulation for the positive monoid of the braid group $B_n$ used by Krammer in \cite{Krammer}.

In Section \ref{sec:2rep} we begin by defining the zigzag algebra $\A$, introduce several gradings on it, and discuss some basic properties of the homotopy category of (graded) projective $\A$ modules.  We then define the relevant baric and $t$-structures on this triangulated category.  Finally, the 2-representation $\Psi:F_n\rightarrow \mbox{Aut}(\D)$ is defined in Section \ref{sec:functors}.

The proof of Theorem \ref{thm:main} is given in Section \ref{sec:pingpong}, and it is by ping pong.  Crucial in the construction is the non-negative 
$\widetilde{o}$-grading of the zigzag algebra arising from a specific orientation $\tilde{o}$ of the doubled complete graph; this grading induces a baric structure on $\D$, and we use the slices of this baric structure to construct the non-empty disjoint sets required for the application of the ping pong lemma.  As a consequence, we obtain a homological realization of the standard word-length metric on $F_n$, which we formulate at the end of the paper in Theorem \ref{thm:metric1}.

In Section \ref{sec:dualpingpong}, we revisit the relationship between the free group and the category $\D$ from the point of view of the $\vec{o}$-grading of the zigzag algebra and the associated baric structure on $\D$.  Here the appearance of the free group as the fundamental group of the punctured disk becomes important.  The main result here is Proposition \ref{prop:dual}, which plays dual ping pong using complexes of projective modules.  (Since they are not needed for the proof of Theorem \ref{thm:main}, a few of the proofs in Section \ref{sec:dualpingpong} are sketched rather than written out completely.) A corollary of Proposition \ref{prop:dual} is that we obtain a homological realization of the word-length metric in the simple Bessis generators of $F_n$; the precise statement is written in Theorem \ref{thm:metric2}. An important part of the setup for dual pingpong is the Hurwitz action of the braid group $B_n$ on reduced reflection expressions for $\gamma= \s_1\dots \s_n$ and a parallel action of $B_n$ on spherical collections in $\D$ (see Section \ref{sec:Hurwitzspherical}).  The action of $B_n$ on spherical collections in $\D$ is similar to (and motivated by) constructions of Bridgeland in \cite{BridgelandCY}.  The final subsection of the paper constructs the exotic metric on $F_n$, explains the little bit we know about it, and poses the question of describing it combinatorially.  

\subsection{Acknowledgements}
We would like to thank Hoel Queffelec for a number of helpful conversations about this paper and our related ongoing collaboration, and the referee for a number of helpful comments.  This paper was written for the proceedings of the conference on Categorification in Algebra, Geometry and Physics, held in Corsica in May, 2015 in honor of the 60th birthday of Christian Blanchet.  We thank the organizers of that conference for the opportunity to speak about this work there, and the editors of this conference proceedings for their patience.

\section{The Free group}\label{sec:free}
In this section we collect some information about the free group.  We recall the definitions of reflections in $F_n$, the positive monoid $F_n^+\subset F_n$ defined by Bessis \cite{Bessis}, and the Hurwitz action of the $n$-strand Artin braid group $B_n$ on reduced reflection expressions for the Coxeter element 
$\gamma = \s_1\dots \s_n.$  

\subsection{The positive monoid of Bessis}

Our presentation in this section closely follows that of \cite{Bessis}, and we refer the reader there for further information regarding the positive monoid and quasi-Garside structure of $F_n$.   We fix $n+1$ distinct points $x_0,x_1,\dots,x_n\in \C$, and set
$$
	F_n := \pi_1(\C-\{x_1,\dots,x_n\},x_0).
$$
It is convenient to choose $x_0 = -1$, and $x_i$ for $i\neq 0$ purely imaginary, with
$$
	-1 < im(x_1) < \dots < im(x_n) < 1,
$$
so that the points $\{x_1,\dots,x_n\}$ lie on a vertical line segment inside the unit disc.  
The unit circle itself, traversed counterclockwise starting at $x_0=-1$, defines an element $\gamma\in F_n$.
We refer to $\gamma$ as a {\it Coxeter element} of $F_n$, since it's image in the Coxeter group $W_n = F_n/\la \s_i^2 \ra$ is a Coxeter element of $W_n$.
In terms of the standard generators $\s_1,\dots,\s_n$, we have $$\gamma = \s_1\s_2\dots \s_n.$$

A {\it connecting path} is a continuous map $f:[0,1] \rightarrow \C$ such that $f(0) = x_0$, $f(1)=x_i$ for some $i=1,\dots,n$, and $t\neq 1\implies f(t)\notin \{x_1,\dots,x_n\}$.  Given a connecting path $f$, we may assign an element $r_f\in F_n$ as follows: begin at $f(0)=x_0$, travel close to $f(1)$ via the connecting path, make a counterclockwise turn around a small circle centred at $f(1)$, and then return to $x_0$ by traversing $f$ backwards.
A {\it $\gamma$-reflection} of $F_n$ is an element $r\in F_n$ such that there is a connecting path $f$ such that $r=r_f$.  For example, the standard generators $\s_1,\dots,\s_n$ of $F_n$ are $\gamma$-reflections, and the corresponding connecting paths are drawn in Figure 1.  (Note that the inverses $\s_1^{-1},\dots,\s_n^{-1}$ are {\it not} $\gamma$-reflections.)  Another connecting path $f$ given in Figure 2;  it corresponds to the $\gamma$-reflection $g^{-1} \s_2 g$, with $g = \s_3^{-1}\s_4\s_3\s_2\s_1$.

\begin{figure}\label{fig:coordinates}
\begin{center}
\scalebox{.75}{\includegraphics{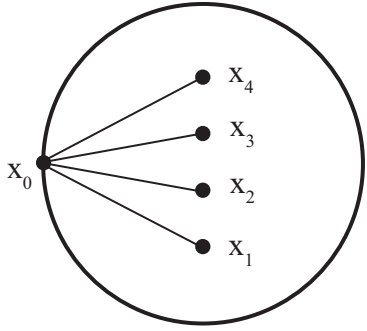}}
\caption{ The standard coordinate system for $F_n$.  }\label{fig:disc1}
\end{center}
\end{figure}

\begin{figure}\label{fig:reflection}
\begin{center}
\scalebox{.75}{\includegraphics{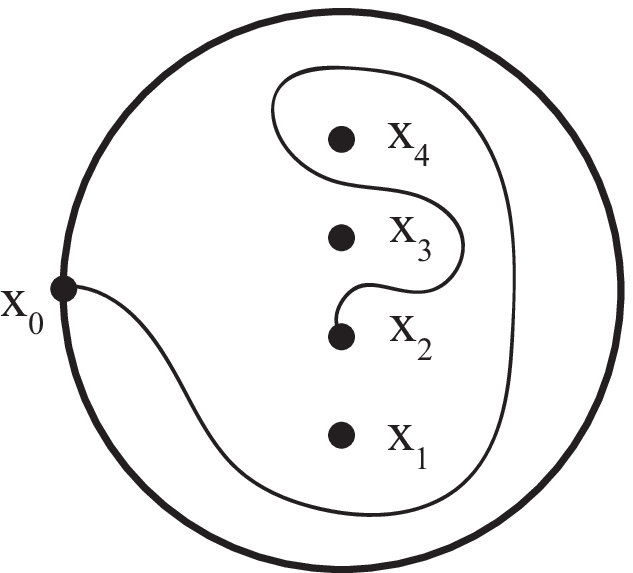}}
\caption{ The $\gamma$-reflection $(\s_1^{-1}\s_2^{-1}\s_3^{-1}\s_4^{-1}\s_3)\s_2(\s_3^{-1}\s_4\s_3\s_2\s_1) $ }\label{fig:disc1}
\end{center}
\end{figure}

We denote by $F_n^+$ the {\it Bessis monoid}, which is the submonoid of $F_n$ generated by the $\gamma$-reflections.  The Bessis monoid $F_n^+$ has a divisibility partial ordering: for $f,g\in F_n^+$,
$$
	f\leq g \iff \exists g\in F_n^+, fh = g.
$$
In fact, there exists $h\in F_n^+$ with $fh=g$ if and only if there exists $h'\in F_n^+$ with $h'f = g$, so the partial order defined above does not depend on the choice of left or right-divisibility.  If $f\leq g$, we say that $f$ {\it divides} $g$.

A slightly larger generating set for $F_n^+$, which includes the $\gamma$-reflections and which plays a prominent role later, is the interval in the Bessis monoid between $1$ and the Coxeter element.  We denote this set by $\Be^+$:
$$
	\Be^+ := \{g \in F_n^+: 1\leq g \leq \gamma\}.
$$
We refer to elements of $ \Be^+$ as {\it simple Bessis elements} of $F_n^+$.

A crucial property of the Bessis monoid $F_n^+$ is that it is a lattice, that is, has greatest common divisors and least common multiples.  In particular, for any element $g\in F_n^+$, there exists a unique left factor $lf(g)\in \Be^+$ such that 
\begin{itemize}
\item the left factor $lf(g)$ divides $g$, and
\item  for all simple Bessis elements $x\in\Be^+$, if $x$ divides $g$, then $x$ divides $lf(g).$
\end{itemize}
The corresponding {\it left-greedy normal form} of $g\in F_n^+$ is defined to be the unique expression for $g$ as a product
$$
	g = y_1y_2\dots y_k
$$
with $y_i\in \Be^+$ and $y_i = lf(y_iy_{i+1}\dots y_k)$ for all $i=1,\dots,k$.

\subsection{The Hurwitz action of $B_n$ on reduced reflection expressions of $\gamma = \s_1\dots \s_{n}$}
For any $(g_1,\dots,g_n)\in F_n^{\times n}$, set
$$
	\tau_i (g_1,\dots,g_n) = (g_1,\dots, g_{i-1}, g_i g_{i+1}g_i^{-1}, g_i,g_{i+2},\dots,g_n).
$$
This defines {\it the Hurwitz action} of the $n$-strand Artin braid group $B_n$ on $F_n^{\times n}$. We are interested in the $B_n$ action on $n$-tuples of $\gamma$-reflections whose product is the Coxeter element $\gamma$.

\begin{definition}
A length $k$ {\it reflection expression} for $\gamma$ is a $k$-tuple $(r_1,\dots,r_k)$ of $\gamma$-reflections whose product
$r_1\dots r_k = \gamma$.  Since $k=n$ is the minimal possible length of a reflection expression for $\gamma$, we refer to a length $n$ reflection expression for $\gamma$ as a {\it minimal length reflection expression.}  We denote by $Red(\gamma)$ the set of minimal length reflection expressions for $\gamma$.
\end{definition}

The following theorem is then due to Artin \cite{Artin}, and we refer to \cite{Bessis} for a proof.
\begin{theorem}\label{thm:Artin}
The Hurwitz action of the Artin braid group $B_n$ on $Red(\gamma)$ is simple and transitive.
\end{theorem}

\section{Ping pong and dual ping pong}\label{sec:ping}
In this section we recall the version of the ping pong lemma which we will use to prove the faithfulness Theorem \ref{thm:main} in Section \ref{sec:pingpong}.  We also make use of the Bessis monoid $F_n^+$ to formulate a dual ping pong lemma, which we will use later in Section \ref{sec:dualpingpong}.

\subsection{Ping pong}
The following lemma, versions of which are ubiquitous in combinatorial and geometric group theory, is usually attributed to Klein.
\begin{lemma}\label{lem:pingpong}  Fix $n\geq 2$.  Let $Y$ be a set, and let 
$$g_1,\dots,g_n: Y \rightarrow Y$$ 
be bijections.  Let $X_1^+,\dots,X_n^+,X_1^-,\dots X_n^-$ be subsets of $Y$, and denote by $X\subset Y$ their union.   
Suppose that
\begin{itemize}
\item the subsets $X_1^+,\dots,X_n^+,X_1^-,\dots X_n^-$ are pairwise disjoint;
\item there exist points $y_1,\dots,y_n \in Y \setminus X$ such that for $j\neq i$, $g_i^{\pm 1}(y_j) \in X_i^{\pm}$, and
\item $x\in X \setminus X_i^{\mp}- \implies g_i(x) \in X_i^{\pm}$.
\end{itemize}
Then the group homomorphism
$$ 
	F_n\rightarrow \{\text{Bijections } g:Y \rightarrow Y\}, \ \ \s_i \mapsto g_i
$$
is injective.
\end{lemma}
\begin{proof}

Let $\s =  \s_{i_k}^{\epsilon_k}\dots \s_{i_1}^{\epsilon_1}$ be a reduced word in $F_n$, so that no $\s_i$ is adjacent to a $\s_i^{-1}$, and let $g = g_{i_k}^{\epsilon_k}\dots g_{i_1}^{\epsilon_1}$ be the image of $\s$ under the homomorphism $\s_i\mapsto g_i$.  We must show that $g$ is not the identity map $1_Y$.

If we can find a point $y\in Y$ such that $g_{i_1}^{\epsilon_1}(y) \in X_{i_1}^{\epsilon_1}$, then it follows 
from the assumptions of the lemma that $g(y) \in X_{i_k}^{\epsilon_k}$.  So if we can find a $y$ such that $g_{i_1}^{\epsilon_1}(y) \in X_{i_1}^{\epsilon_1}$ but $y \notin X_{i_k}^{\epsilon_k}$, we will have $g(y) \neq y$.  Now the assumption on the points $y_j\in Y\setminus X$ guarantees that for $j\neq i_1$,
$$
	g(y_j) \neq y_j.
$$ 
Thus that $g\neq 1$.
\end{proof}

\subsection{Dual Ping pong}

The Bessis monoid $F_n^+$ has the pleasant feature that any element $\b\in F_n$ can be written
$$
	\beta = \alpha_1^{-1} \alpha_2, 
$$
where $\alpha_1,\alpha_2 \in F_n^+$ are in the positive Bessis monoid.  This implies the following.

\begin{lemma}\label{lem:dual}
Suppose that $\psi: F_n \rightarrow  \{\text{Bijections } g:Y \rightarrow Y\}$ is an action of $F_n$ on a set $Y$ which distinguishes elements of the Bessis monoid $F_n^+$.  Then $\psi$ is injective.
\end{lemma}
\begin{proof}
Suppose that $\psi(\b) = 1_Y$ for some $\b\in F_n$.  Writing $\beta = \alpha_1^{-1} \alpha_2$
with $\alpha_1,\alpha_2 \in F_n^+$, we see that $\psi(\alpha_2)=\psi(\alpha_1)$.  But by assumption, $\psi$ distinguishes elements of $F_n^+$, so that $\alpha_1 = \alpha_2$.  Thus $\b = 1$.
\end{proof}

Lemma \ref{lem:dual} shows that one can reduce the question of whether or not an action of the free group is faithful to a question about whether or not the action distinguishes elements of the Bessis monoid.  As a result, a ``dual" analogue of the ping pong Lemma \ref{lem:pingpong} can be formulated using only the simple Bessis elements $\Be^+ = \{w: 1\leq w\leq \gamma\}$ .  More precisely, we have the following, which we refer to as a dual ping pong lemma for the free group.

\begin{lemma}\label{lem:dualpingpong}
Let $Y$ be a set, and let $\psi: F_n \longrightarrow \{\text{Bijections }g: Y \rightarrow Y \}$ be an action.  Denote by $\psi_\b$ the image of an element $\b\in F_n$ under this homomorphism.  Suppose that there exist non-empty disjoint subsets 
$\{X_w\}_{w\in \cal Be^+}$ of $Y$, such that for all simple Bessis elements $u,w\in \Be^+$,
$$
	\psi_u ( X_w) \subset X_{lf(uw)}.
$$
Then the group homomorphism $\psi$ is injective.
\end{lemma}
\begin{proof}
Suppose that $\b\in F_n^+$.  Let $\b = u_1\dots u_k$ be the left greedy normal form of $\b$.  We will show that the left factor $u_1$ of $\b$ is determined by the bijection $\psi_\b$.  Once we show this, it follows by induction that the entire left greedy normal form of $\b$ can be read off from the bijection $\psi_\b$; in particular, $\psi$ will distinguish elements of the Bessis monoid $F_n^+$, and the statement then follows from Lemma \ref{lem:dual}.

To see that $u_1$ is determined by the bijection $\psi_\b$, choose $x\in X_{u_k}$, where $u_k$ is the rightmost term in the left greedy normal form of $\b$.  By the assumptions of the lemma, we have $\psi_\b(x)\in X_{u_1}$, so that $\psi_\b$ determines $u_1$.
\end{proof}
\begin{remark}\label{rem:simplification}
One can also show that in checking the condition 
$$\psi_u ( X_w) \subset X_{lf(uw)}$$
for all simple Bessis elements $u,w\in \Be^+$, it suffices to check that $\psi_t(X_w) \subset X_{lf(tw)}$ for all $w\in \Be^+$ and $t$ a $\gamma$-reflection.
\end{remark}

\section{A 2-representation the free group}\label{sec:2rep}
The goal of this section is to define a 2-representation of $F_n$.  The underlying category of the 2-representation will be the homotopy category of projective modules over a finite dimensional algebra $\A$.  The algebra $\A$ is itself a quotient of the the path algebra of a quiver, so we begin by introducing the required quiver notation.

\subsection{Quivers associated to the doubled complete graph}\label{sec:algebras}
Let $Q$ denote the doubled complete graph on $n$ vertices.  We number the vertices of $Q$ by the integers $1,2,\dots,n$.  In the doubled complete graph, for any pair of distinct vertices $i,j$ there are two (unoriented) edges between $i$ and $j$.  To distinguish these two edges from each other, we will name one of them the $x$-edge from $i$ to $j$ and the other the $y$ edge from $i$ to $j$.  The graph $Q$ for $n=4$ is depicted below.


$$
\begin{tikzpicture}
  \SetGraphUnit{5}
  \Vertex{2}
  \NOWE(2){1}
  \NOEA(2){3}
	\NOEA(1){4}
  \Edge(1)(2)
  \Edge(2)(1)
  \Edge(3)(2)
  \Edge(2)(3)
  \Edge(1)(3)
  \Edge(3)(1)
 \Edge(1)(4)
  \Edge(4)(1)
 \Edge(4)(3)
  \Edge(3)(4)
 \Edge(2)(4)
  \Edge(4)(2)
  \tikzset{EdgeStyle/.append style = {bend left = 20}}
\end{tikzpicture}
$$

\subsection{The zigzag algebra $\widehat{A}$}
We now fix a field $\k$ of characteristic 0.  For a quiver $X$, we denote by $\mbox{Path}(X)$ the path algebra over $\k$.

Let $\widehat{Q}$ be the Ginzburg double of $Q$: in the quiver $\widehat{Q}$, each unoriented edge of $Q$ has been replaced by two oriented edges, one in each direction; thus, for each $i\neq j$ there are now four oriented edges between $i$ and $j$, two of them pointing $i \to j$ and two pointing $j\to i$.  In addition, we put a loop, labeled $z$, at each vertex.  For each pair of vertices $i,j$ with $i<j$, we label the edges from $i\to j$ by $x$ and $y$, and the edges pointed $j \to i$ by $x^*$ and $y^*$.  The quiver $\widehat{Q}$ for $n=3$, with its edge labels, is drawn below.

$$
\begin{tikzpicture}
  \tikzset{EdgeStyle/.append style = {->,bend left = 8}}
  \SetGraphUnit{5}
  \Vertex{2}
  \NOWE(2){1}
  \NOEA(2){3}
  \Edge[label = x](1)(2)
  \Edge[label = x*](2)(1)
  \Edge[label = x*](3)(2)
  \Edge[label = x](2)(3)
  \Edge[label = x](1)(3)
  \Edge[label = x*](3)(1)
    \tikzset{EdgeStyle/.append style = {->,bend left = 25}}
      \Edge[label = y](1)(2)
  \Edge[label = y*](2)(1)
  \Edge[label = y*](3)(2)
  \Edge[label = y](2)(3)
  \Edge[label = y](1)(3)
  \Edge[label = y*](3)(1)
  
 \Loop[dist = 4cm, dir = NO, label = z](1.west)
  \Loop[dist = 4cm, dir = WE, label = z](2.south)
   \Loop[dist = 4cm, dir = SO, label = z](3.east)

\end{tikzpicture}
$$

The main algebra of interest here is a finite-dimensional quotient of the path algebra  $\mbox{Path}_\k(\widehat{Q})$ of the Ginzburg quiver $\widehat{Q}$.  
The length 0 path at vertex $i$ in $\mbox{Path}_\k(\widehat{Q})$ is denoted by $e_i$, and this path is an idempotent in $\mbox{Path}_\k(\widehat{Q})$.  When regarded as elements of the path algebra, we somewhat abusively refer to edges of $\widehat{Q}$ by their labels.  So, for example, the four edges connecting vertices 1 and 2 will be referred to as $x,x^*,y$ and $y^*$, as will the four edges connecting vertices 2 and 3, while each of the loops will be denoted by $z$.  As a result, in order for a string of letters in the alphabet $\{x,y,x^*,y^*,z\}$ to give a well-defined path, we should specify the starting and ending vertices of the path.

We define
$$
	\widehat{A} = \mbox{Path}_\k(\widehat{Q}) / I,
$$
where $I$ is the two-sided ideal in $\mbox{Path}_\k(\widehat{Q})$ generated by
\begin{itemize}
\item all length two paths which start and end at distinct vertices;
\item the following linear combinations of length two paths which start and end at the same vertex:
$$
	e_i a^*b e_i - \delta_{a,b}e_i  ze_i , \ \  a,b\in \{x,y,x^*,y^*\}.
$$
\end{itemize}
In the last relation it is understood that ${x^*}^* = x$ and ${y^*}^* = y$. 

A few words about these relations:
first, note that as a consequence of the second relation we could have presented the algebra $\A$ as a quotient of the path algebra of the quiver without using the loops $z$ at all: for example, the loop $e_1ze_1$ at vertex 1 is equal to the length 2 path $e_1xx^*e_1$ starting and ending at vertex 1 and passing through vertex 3; second, since all of the length two relations involving a path $i \to j \to k$ do not depend on the middle vertex $j$, we have also omitted notation for this middle vertex in the relations above.  Finally, a special case of the last relation says that a length two path of the form $e_ix^*ye_i$ starting and ending at vertex $i$ -- regardless of which other vertex it travels through -- is zero.  (Similarly, the length two paths $e_ixy^*e_i,\ e_iyx^*e_i$ and $e_iy^*xe_i$ are zero.)  The loop $e_iz e_i$ at vertex $i$ is non-zero, but it follows from relations that any path through $i$ which involves both $z$ and another edge in the quiver must be equal to zero in $\A$.
 
The case $n=1$ is slightly degenerate, since when $n=1$ there are no edges in $Q$; our convention when $n=1$ is to set $\widehat{A} \cong \k[z]/z^2$.  By analogy with the case of simply-laced quivers considered in \cite{HuerfanoKhovanov}, we call the algebra $\widehat{A}$ the {\it zigzag algebra} of $Q$.

The following lemma is immediate from the definition of the algebra $\widehat{A}$.  (Essentially, the relations in $\A$ were chosen so that the lemma below would hold.)
\begin{lemma}\label{lem:pathspaces}
For $i<j$:
$$
	e_i\widehat{A} e_j = \mbox{span}_{\k}\{x,y\} \cong \k^2,
$$
$$
	e_j\widehat{A} e_i = \mbox{span}_{\k}\{x^*,y^*\} \cong \k^2,
$$
and
$$
	e_i\widehat{A} e_i = \mbox{span}_\k\{e_i,z\} \cong \k^2.
$$
In particular, $\widehat{A}$ is a finite-dimensional $\k$-algebra of dimension $2n^2$.
\end{lemma}

The left $\A$ module $P_i = \A e_i$ is indecomposable and projective, and the $n$ modules $\{P_i\}_{i=1}^n$ form a complete collection of pairwise distinct isomorphism classes of indecomposable projective left $\A$ modules.  From Lemma \ref{lem:pathspaces}, we see that the endomorphism algebra of $P_i$ is given by
$$
	\mbox{End}_{\A}(P_i,P_i) \cong e_i\A e_i \cong \k[z]/z^2,
$$
While all the other hom spaces $\Hom_{\A}(P_i,P_j) \cong e_i \widehat{A} e_j$ are also two-dimensional vector spaces with a distinguished basis.  

\subsection{Gradings on $\A$} \label{sec:gradings}
In the sections to come we will consider several different non-negative gradings on $\A$.  Of primary importance are
\begin{itemize}
\item the path length grading, given by setting
$$
\deg_{path}(e_i) = 0, \ \ \deg_{path}(z) = 2,
$$
$$
	\deg_{path}(x) = \deg_{path}(x^*) = \deg_{path}(y) = \deg_{path}(y^*) = 1;
$$
\item the orientation grading $\deg_o$ corresponding to an orientation $o$ of the graph $Q$.  A choice of orientation $o$ of $Q$ is a choice of head and tail for each edge of $Q$, and choosing such an orientation gives us a quiver $Q_o$.  The quiver $Q_o$ gives rise to a grading on $\A$ by setting, for $i<j$,
$$
	\deg_o(x) = \begin{cases}
	1, \text{ if the } x \text{ edge of } Q \text{ between }i \text{ and }j \text{ goes from } i\to j \text{ in } Q_o;\\
	0, \text{ otherwise.}
	\end{cases}
$$
$$
\deg_o(y) = \begin{cases}
	1, \text{ if the } y \text{ edge of } Q \text{ between }i \text{ and }j \text{ goes from } i\to j \text{ in } Q_o;\\
	0, \text{ otherwise.}
	\end{cases}
$$
$$
	\deg_o(x^*) = 1-\deg_o(x);
$$
$$
 	\deg_o(y^*) = 1-\deg_o(y).
$$
$$
	\deg_o(z) = 1.
$$
\end{itemize}
These all non-negative integral gradings on the algebra $\A$.  For an orientation $o$ of $Q$, let $o^\vee$ be the opposite orientation.
The degree 0 subalgebras of $\A$ for various gradings can be described as follows.
\begin{itemize}
\item When endowed with the path-length grading, the degree $0$ subalgebra of $\A$ is semi-simple and isomorphic to $\k^n = \mbox{span}_{\k}\{e_1,\dots,e_n\}$.
\item When endowed with the $o$-orientation grading, the degree $0$ subalgebra of $\A$ is isomorphic to the quotient of the path algebra $\mbox{Path}_\k(Q_{o^\vee})$ by the two-sided ideal generated by all length two paths in $Q_{o^\vee}$.
\end{itemize}

We introduce special notation for the two specific orientations that will play an important role later:
\begin{itemize}
\item  Let $\widetilde{o}$ denote the orientation of $Q$ obtained by orienting one edge (the $x$ edge) between $i$ and $j$ as an arrow $i \to j$ and orienting the other edge (the $y$ edge) of $Q$ as an arrow $j \to i$.  We denote by $\widetilde{Q} = Q_{\widetilde{o}}$ the associated quiver, given by endowing $Q$ with the orientation $\widetilde{o}$.  The quiver $\widetilde{Q}$ for $n=4$ is drawn below.  $$
\begin{tikzpicture}
  \tikzset{EdgeStyle/.append style = {->}}
  \SetGraphUnit{5}
  \Vertex{2}
  \NOWE(2){1}
  \NOEA(2){3}
	\NOEA(1){4}
  \Edge[label = x](1)(2)
  \Edge[label = y*](2)(1)
  \Edge[label = x](2)(3)
  \Edge[label = y*](3)(2)
  \Edge[label = x](1)(3)
  \Edge[label = y*](3)(1)
 \Edge[label = x](1)(4)
  \Edge[label = y*](4)(1)
 \Edge[label = x](3)(4)
  \Edge[label = y*](4)(3)
 \Edge[label = x](2)(4)
  \Edge[label = y*](4)(2)
\end{tikzpicture}
$$
The edge labels on the above quiver specify the natural embedding of  $\widetilde{Q}$ as a sub quiver of the Ginzburg double $\widehat{Q}$.  In the zigzag algebra $\A$, the edges of the Ginzburg double $\widehat{Q}$ which are in the image of this embedding (the $x$ and $y^*$ edges) will have $\widetilde{o}$-degree 1 in the zigzag algebra $\A$, whereas the dual edges ($x^*$ and $y$) will have $\widetilde{o}$-degree 0.

\item Let $\vec{o}$ denote the orientation of $Q$ such that, for $i<j$, both edges between $i$ and $j$ are oriented to point $i\to j$ (thus the quiver $\vec{Q} = Q_{\vec{o}}$ has no oriented cycles). The quiver $\vec{Q}$ for $n=4$ is drawn below.
$$
\begin{tikzpicture}
  \tikzset{EdgeStyle/.append style = {->}}
  \SetGraphUnit{5}
  \Vertex{2}
  \NOWE(2){1}
  \NOEA(2){3}
	\NOEA(1){4}
  \Edge[label = x](1)(2)
   \Edge[label = x](2)(3)
 
  \Edge[label = x](1)(3)

 \Edge[label = x](1)(4)
 \Edge[label = x](3)(4)

 \Edge[label = x](2)(4)

    \tikzset{EdgeStyle/.append style = {<-}}
     \Edge[label = y](2)(1)
      \Edge[label = y](3)(2)
  \Edge[label = y](3)(1)
    \Edge[label = y](4)(1)
      \Edge[label = y](4)(3)
       \Edge[label = y](4)(2)

\end{tikzpicture}
$$
The edge labels on the above quiver specify the embedding of  $\vec{Q}$ as a sub quiver of the Ginzburg double $\widehat{Q}$.  In the zigzag algebra $\A$, the edges of the Ginzburg double $\widehat{Q}$ which are in the image of this embedding (the $x$ and $y$ edges) will have $\vec{o}$-degree 1, whereas the dual edges ($x^*$ and $y^*$) will have $\vec{o}$-degree 0.

\end{itemize}

\subsection{Categories of complexes}
We denote by $\D$ the homotopy category of bounded complexes of finitely-generated projective $\A$-modules.  An object of $\D$ is a bounded complex of finitely generated projective modules
$$
	Y = (Y^m,\partial^m),\ \ \partial^m:Y^m \rightarrow Y^{m+1},\ \  \partial^{m+1}\circ \partial^m = 0.
$$
A morphism $f$ from $X$ to $Y$ is a collection of $\A$ module maps $f^i:X^i \rightarrow Y^i$ intertwining the differentials.  Two maps $f,g:X\rightarrow Y$ are equal in $\D$ if they are homotopic when regarded as maps of chain complexes.
We let $[k]$ denote the auto-equivalence of $\D$ which shifts a complex $k$ degrees to the left: 
$$
	Y[k]^m = Y^{k+m}, \ \ \ \partial_{Y[k]} = (-1)^k\partial_Y. 
$$
The pair $(\D,[1])$ is a finitely-generated, $\k$-linear triangulated category.

Given a map of complexes $f : X\rightarrow Y$, the cone of $f$ is the complex $X[1] \oplus Y$ with the differential 
$$
	\partial (x,y) = (-\partial_{X}(x),f(x)+\partial_{Y}(y)).
$$

In the sequel we will often endow $\A$ with one of the non-negative gradings defined earlier, in which case we will be interested in the homotopy category of bounded complexes of finitely-generated projective {\it graded} $\A$-modules.  Since different gradings will be used in different sections, we will state explicitly at the beginning of each section which grading we are considering.  Abusing notation slightly, we will still use $\D$ to denote the homotopy category of bounded complexes of finitely-generated graded modules in each of these sections; $[k]$ always denotes a shift in homological degree; we will always denote by $\la k \ra$ the shift in the relevant internal grading.

\subsection{Minimal complexes}
Let $Y\in \D$ be an indecomposable complex, so that 
$$Y \cong Y_1\oplus Y_2 \in \D \implies Y_1\cong 0\text{ or }Y_2\cong 0.
$$  
Let $\tilde{Y}$ be a representative of $Y$ in the $\k$-linear category of bounded complexes of projective $\A$ modules.  In other words, $\tilde{Y}$ is a bounded complex of $\A$ modules which is homotopy equivalent to $Y$, but not necessarily isomorphic to $Y$ as a chain complex. We say that $\tilde{Y}$ is a {\it minimal complex} for $Y$ if $\tilde{Y}$ is indecomposable in the additive category of bounded complexes of projective $\A$ modules.  In particular, if $\tilde{Y}$ is a minimal complex for $Y$, then when regarded as a complex of projective $\A$ modules, $\tilde{Y}$ has no contractible summands.  Moreover, the chain groups of a minimal complex are themselves determined up to isomorphism as projective $\A$ modules.  We thus refer to $\tilde{Y}$ as {\it the} minimal complex of $Y$, with the understanding that the chain groups of $\tilde{Y}$ are only determined up to isomorphism.  When $Y$ is not necessarily indecomposable, the minimal complex of $Y$ is defined to be the direct sum of the minimal complexes of the indecomposable summands of $Y$.  As every complex $Y$ is homotopy equivalent to a minimal complex, in some arguments it will be convenient to study $Y$ by assuming it is already minimal.

\subsection{Slicing complexes} \label{sec:slice}
The goal of this section is to briefly describe two ways to use a grading on $\A$ to ``slice" objects of $\D$ into homogeneous pieces.  The first way, which is very well known, is via a bounded $t$-structure.  We assume some familiarity with the notion of a $t$-structure here, and refer the reader to \cite{GelfandManin} for further information.  For now, it suffices to note that a bounded $t$-structure on $\D$ defines a collection of abelian subcategories $\{\D^k\}_{k\in \Z}$.  The subcategory $\D^0$ is called the {\it heart} of the $t$-structure, and a bounded $t$-structure on $\D$ is determined by its heart.  Moreover, a collection of full $\k$-linear subcategories $\{\D^k\}_{k\in \Z}$ determines a $t$-structure provided the following conditions hold:
\begin{itemize}
\item [(a)] for all $k\in \Z$, $\D^k[1] = \D^{k+1}$,
\item[(b)]if $X\in \D^k$ and $Y\in \D^l$ with $k>l$, then $\Hom_\D(X,Y)=0$, and
\item[(c)]
for every nonzero object $E\in\D$ there are integers $m<n$ and a collection
of triangles
\[
\xymatrix@C=.2em{ 0_{\ } \ar@{=}[r] & E(m) \ar[rrrr] &&&& E(m+1) \ar[rrrr]
\ar[dll] &&&& E(m+2) \ar[rr] \ar[dll] && \ldots \ar[rr] && E(n-1)
\ar[rrrr] &&&& E(n) \ar[dll] \ar@{=}[r] &  E^{\ } \\
&&& Y(m+1) \ar@{-->}[ull] &&&& Y(m+2) \ar@{-->}[ull] &&&&&&&& Y(n)
\ar@{-->}[ull] }
\]
with $Y(k)\in\D^k$ for all $k$.
\end{itemize}
We will refer to the terms $Y(k)$ appearing above as the {\it t-slices} of the object $E$.

A non-negative grading on the algebra $\A$ gives rise to a bounded $t$-structure on the homotopy category of graded $\A$ modules $\D$ whose heart is the category of {\it linear complexes of projectives.}  By definition, a linear complexes of projective modules is a complex $Y\in \D$ such that the term $Y^k$ in homological degree $k$ is a direct sum of projective modules $\oplus P_j\la k \ra$, so that the homological degree matches the internal grading shift on the nose. Thus gradings on $\A$ give rise to $t$-structures, and hence to ``slicings" of complexes of graded projective $\A$ modules.  We refer the reader to \cite{MOS} for a detailed discussion of linear complexes of projective modules over finite dimensional non-negatively-graded algebras.  

Besides using $t$-structures, there is another somewhat simpler way to slice an object of $\D$ into homogeneous pieces, which is to ignore the homological degree completely and slice the minimal complex of $E$ using only an internal grading on $\A$-modules.  The resulting decomposition of can be formalised in the language of {\it baric structures}, introduced by Achar and Treumann in \cite{AcharTreumann}.  

In more detail, fix an orientation $o$ of $Q$ and let $\D$ denote the homotopy category of complexes of $o$-graded $\A$-modules, where $\A$ is endowed with the $o$-grading defined in Section \ref{sec:gradings}.  For $k\in \Z$, let $\D_k$ denote the full subcategory of $\D$ consisting of complexes $Y = (Y^m,\partial^m)$ such that, in the minimal complex of $Y$, the chain groups are direct sums of projective modules whose internal grading shift is exactly $k$ in every homological degree:
$$
	Y^m \cong \oplus_i P^{\oplus n_i}_i\la k \ra \text{ for all } m\in \Z.
$$
The subcategories $\{\D_k\}$ define a 
baric structure in the sense of \cite{AcharTreumann}.  Since the baric structures we consider are essentially just gradings on projective $\A$-modules, rather than give the full abstract definition of a baric structure on a triangulated category here, we prefer to simply note that the subcategories $\{\D_k\}$ defined above satisfy the following properties (all of which follow directly from properties of the non-negative grading on $\A$):

\begin{itemize}
\item [(a)] for all $k\in \Z$, $\D_k[1] = \D_k$,
\item[(b)]if $X\in \D_k$ and $Y\in D_l$ with $k>l$, then $\Hom_\D(X,Y)=0$, and
\item[(c)]
for every non-zero object $E\in\D$ there are integers $m<n$ and a collection
of triangles
\[
\xymatrix@C=.2em{ 0_{\ } \ar@{=}[r] & E(m) \ar[rrrr] &&&& E(m+1) \ar[rrrr]
\ar[dll] &&&& E(m+2) \ar[rr] \ar[dll] && \ldots \ar[rr] && E(n-1)
\ar[rrrr] &&&& E(n) \ar[dll] \ar@{=}[r] &  E_{\ } \\
&&& Y(m+1) \ar@{-->}[ull] &&&& Y(m+2) \ar@{-->}[ull] &&&&&&&& Y(n)
\ar@{-->}[ull] }
\]
with $Y(k)\in\cal \D_k$ for all $k$.
\end{itemize}
We refer to the subcategories $\{\D_k\}$ induced from the $o$-grading on $\A$ as the {\it $o$-baric structure} on $\D$, and to the objects $Y(k) \in \D_k$ appearing above as the $o$-{\it baric slices} of the object $E$.  

Given an interval $[k,l]\subset \Z$, we set 
$\D_{[k,l]}$ to be the full subcategory of $\D$ consisting of those complexes $Y$ whose non-zero baric slices live in $\D_m$ for $k\leq m \leq l$.  Similarly, $\D^{[k,l]}$ is the full subcategory of $\D$ whose non-zero $t$-slices live in $\D^m$ for $k\leq m \leq l$.  The subcategories $\D_{\leq k}$, $\D_{\geq k}$, $\D^{\leq k}$, $\D^{\geq k}$ are defined similarly.

We have chosen to use the same notation $Y(k)$ for both baric slices and for $t$-slices; with the exception of a brief point in the last section of paper we will on make use of baric slices, and in any case it will always be made clear in each section which structures -- baric or $t$ -- we are considering.  Despite some similarity in the properties shared by baric and $t$-structures, there are important differences between the two notions.  For instance, in a $t$-structure, the heart $\D^0$ is an abelian subcategory of the triangulated category $\D$, and this subcategory is not preserved by the homological shift $[1]$; by contrast, the baric heart $\D_0$ will not in general be an abelian subcategory of $\D$, though in the cases we consider it will be closed under homological shifts.

\subsection{Top and bottom slices}
Fix an orientation $o$ of $\D$, and let $Y\in \D$.  Using the $o$-baric slices $\{Y(k)\}_{k\in \Z}$, we may define integers
$\phi_{+}(Y)$ by
$$
	\phi_{+}(Y) = sup \{k\in\Z : Y(k) \neq 0\} 
$$
and $\phi_{-}(Y)$ by
$$
	\phi_{-}(Y) = inf \{k\in\Z : Y(k) \neq 0\}.
$$
As the minimal complex of $Y$ is a bounded complex of finitely-generated graded $\A$-modules, the chain groups of the minimal complex are all supported in finitely many $o$-degrees, so that slices $Y(k)$ will be non-zero for only finitely many $k$.  We may also define complexes $Y(\phi_+)$ and $Y(\phi_-)$ -- the {\it top and bottom slices of $Y$} -- by
$$
	Y(\phi_+) := Y(\phi_{+}(Y) ), \ \ 
	Y(\phi_-) := Y(\phi_{-}(Y)).
$$
These top and bottom slices will play an essential role in the ping pong and dual ping pong constructions in latter sections.

\subsection{The functors $\Sig_i$}\label{sec:functors}
In this section we define our main object of interest, the 2-representation
$$	
	\Psi: F_n \longrightarrow \mbox{Aut}(\D).
$$

Define homomorphisms $f_i,g_i$, $1\leq i \leq n$ of $(\A,\A)$ bimodules
$$
	f_i: \A e_i\otimes_{\k} e_i\A \longrightarrow \A,
$$
$$
	g_i: \A \longrightarrow \A e_i\otimes_\k e_i\A
$$
by
$$
	f_i : e_i\otimes e_i \mapsto e_i,
$$
$$
	g_i: 1 \mapsto z\otimes e_i + e_i \otimes z + \sum x\otimes x^* + x^*\otimes x + y\otimes y^* + y^*\otimes y,
$$
where $z$ denotes the loop at $i$ and the sum is over all $x,x^*,y,y^*$ edges out of vertex $i$.

Let $\Sig_i$ denote the two-term complex of $(\A,\A)$ bimodules
$$
\Sig_i:= 	\A \xrightarrow{g_i} \A e_i \otimes_\k e_i\A,
$$
with the term $\A$ sitting in homological degree $0$.  Let $\Sig_i^{-1}$ denote the two-term complex of $(\A,A)$ bimodules
$$
	\Sig_i^{-1} := \A e_i \otimes_\k e_i \A \xrightarrow{f_i} \A,
$$
with the term $\A$ sitting in homological degree $0$.
The functors
$$
	Y \mapsto \Sig_i\otimes_{\A} Y, \ \ \ Y \mapsto \Sig_i^{-1}\otimes_{\A} Y,
$$ 
are endofunctors of $\D$; we somewhat abusively denote these endofunctors by $\Sig_i$, $\Sig_i^{-1}$, too.

The proof of the following proposition follows exactly as in Proposition 2.4 of \cite{KhovanovSeidel} (see also \cite{HuerfanoKhovanov}).
\begin{proposition}\label{prop:2-rep}
For each $i = 1,\dots,n$, the functors $\Sig_i$ and $\Sig_i^{-1}$ are mutually inverse equivalences of $\D$, i.e., there are functor isomorphisms
$$
	\Sig_i \Sig_i^{-1} \cong 1_\D \cong \Sig_i^{-1} \Sig_i.
$$
Thus the assignment
$$
	\Psi: F_n \longrightarrow \mbox{Aut}(\D), \ \ \ \sigma_i \mapsto \Sig_i, \ \ \ \s_i^{-1}\mapsto \Sig_i^{-1},
$$
is a (weak) 2-representation of the free group $F_n$.
\end{proposition}
An important point which emerges in the proof of the above proposition is that the functor isomorphisms 
$$
	\Sig_i \Sig_i^{-1} \cong 1_\D \cong \Sig_i^{-1} \Sig_i
$$
will not hold if $\D$ is replaced by the additive category of complexes of finitely-generated projective $\A$-modules; these isomorphisms only hold after passing to the homotopy category.

When we regard $\A$ as a graded algebra, the bimodule map $g_i$ used in the definition of the auto-equivalence $\Sig_i$ is not a map of degree 0, but rather a map of degree $1$ for the $o$-grading for an orientation $o$ and a map of degree 2 for the path-length grading.  Thus, in order to define auto-equivalences of the homotopy category of graded modules, we need to add a grading shift to one of the bimodules used in the definition of $\Sig_i$.  When we consider $\A$ as a graded algebra with the $o$-grading coming from an orientation, the auto-equivalence $\Sig_i$ will be defined as tensoring with the complex 
$$
	\A \xrightarrow{g_i} (\A e_i \otimes_\k e_i\A)\la 1 \ra,
$$
while when we consider $\A$ as a graded algebra with the path-length grading, $\Sig_i$ will be given by tensoring with
$$
	\A \xrightarrow{g_i} (\A e_i \otimes_\k e_i\A)\la 2 \ra.
$$
With these additional grading shifts, the statement of Proposition \ref{prop:2-rep} also holds  when $\Sig_i$ and $\Sig_i^{-1}$ are regarded as endofunctors of the homotopy category of graded projective $\A$ modules. 

We record here basic example calculations of the complexes $\Sig_i(P_j)$, taking into account the different possible gradings.
\begin{itemize}\item If we endow $\A$ with the symmetric $\widetilde{o}$ grading and let $P_i$ denote the $\widetilde{o}$-graded projective $\A$ module with $e_i$ sitting in $\widetilde{o}$-degree 0, then for $i\neq j$, 
$$
	\Sig_i(P_j) \cong \big( P_j \rightarrow P_i\la 1 \ra\oplus  P_i\big) \text{ and } \Sig_i(P_i)\cong P_i\la 1 \ra [-1].
$$
\item If we endow $\A$ with the ordered $\vec{o}$-grading and now let $P_i$ denote the $\vec{o}$-graded projective $\A$ module with $e_i$ sitting in degree 0, then 
$$
	\Sig_i(P_i)\cong P_i\la 1 \ra [-1];
$$
for $i> j$,
$$
	\Sig_i(P_j) \cong \big(P_j \rightarrow P_i\la 1 \ra \oplus P_i\la 1 \ra\big);
$$
for $i<j$,
$$
	\Sig_i(P_j) \cong \big(P_j \rightarrow P_i\oplus P_i\big).
$$
\item If we endow $\A$ with the path-length grading and let $P_i$ denote the path-graded projective $\A$ module with $e_i$ sitting in degree 0,
then we have
$$
	\Sig_i(P_j) \cong \big(P_j \rightarrow P_i \la 1 \ra \oplus P_i\la 1 \ra \big) \text{ and } \Sig_i(P_i)\cong P_i\la 2 \ra [-1]
$$
for all $i\neq j$.
\end{itemize}

\section{2-linearity via ping pong and the $\widetilde{o}$-grading on $\A$}\label{sec:pingpong}
In this section, we prove Theorem \ref{thm:main}, which says that the group homomorphism
$$
\psi: F_n\longrightarrow [\mbox{Aut}(\D)]
$$
is injective.  {\it In the entirety of this section, we endow all algebras, modules, and complexes with the $\widetilde{o}$-grading determined by the symmetric orientation $\widetilde{o}$ on the doubled complete graph $Q$.}  In particular, in this section $\D$ will denote the homotopy category of bounded complexes of projective $\widetilde{o}$-graded 
$\A$ modules. The main ingredients in the proof of Theorem \ref{thm:main} are the symmetric orientation $\widetilde{o}$, the top and bottom $\widetilde{o}$-baric slices 
$Y(\phi_+),Y(\phi_-)$ of $Y\in\D$ defined in Section \ref{sec:slice}, and the ping pong Lemma \ref{lem:pingpong}.

\subsection{The sets $X_i^{\pm}$} 
For $i\in I$, we set 
$$
	\P_i = \bigoplus_{k\in \Z} P_i[k].
$$
Thus, for $m\in \Z$, the projective module $\P_i\la m\ra$ is the direct sum of all possible shifts of $P_i$ which lie in the subcategory $\D_m$ for the $\widetilde{o}$ baric structure on $\D$.  Technically, the module $\P_i$ is not an object of $\D$, as it is not homotopy equivalent to a bounded complex of finitely-generated modules.  Nevertheless, the summands of $\P_i$ are all objects of $\D$, and for $Y\in \D$, the hom space
$$
	\Hom(Y, \P_i) = \bigoplus_{k\in \Z} \Hom(Y, P_i[k])
$$
is a well-defined finite-dimensional vector space.

The following is the key definition in the proof of Theorem \ref{thm:main}.
For $i=1,\dots,n$, we define subsets $X_i^{+}$ and $X_i^-$ of objects of $\D$ by declaring that
\begin{enumerate}
\item $X_i^+$ consists of those indecomposable complexes $Y$ such that
\begin{itemize}
\item the top slice $Y(\phi_+)$ is isomorphic to a direct sum of shifts of $P_i$;
\item $\Hom(Y, \P_j\la \phi_-(Y)\ra) = 0  \iff j= i$;
\end{itemize}

\item $X_i^-$ consists of those indecomposable complexes $Y$ such that
\begin{itemize}
\item The bottom slice $Y(\phi_-)$ is isomorphic to a direct sum of shifts of $P_i$;
\item $\Hom(\P_j\la \phi_+(Y)\ra,Y ) = 0  \iff j= i.$
\end{itemize}
\end{enumerate}

We now check that the sets $X_i^\pm$ and the points $z_i=P_i$ satisfy the assumptions of the ping pong Lemma \ref{lem:pingpong}.

\begin{proposition}\label{prop:pingpong} We have the following:
\begin{enumerate}
\item for all $j\neq i$, $\Sig_i^{\pm1} (P_j) \in X_i^{\pm}$;
\item the sets $X_1^-,\dots,X_n^-,X_1^+,\dots,X_n^+$ are pairwise disjoint;
\item for all $i\neq j$, $Y\in X_j^-\implies \Sig_i(Y) \in X_i^+$;
\item for all $i\neq j$,  $Y\in X_j^+\implies \Sig_i^{-1}(Y) \in X_i^-$;
\item for all $i,j$, $Y\in X_j^+\implies \Sig_i(Y) \in X_i^+$;
\item for all $i,j$, $Y\in X_j^-\implies \Sig_i^{-1}(Y) \in X_i^-$
\end{enumerate}
\end{proposition}
\begin{proof}
The fact that $\Sig_i(P_j) \in X_i^+$ when $j\neq i$ follows by direct computation, as
$$
	\Sig_i(P_j) \cong \big( P_j \rightarrow P_i\la 1 \ra\oplus  P_i\big),
$$
and the properties required for membership in $X_i^+$ are easily checked for this complex.
Similarly, when $j\neq i$, it is straightforward to see that $\Sig_i^{-1}(P_j)\in X_i^-$.  It is also clear from the definition of the sets $X_i^+$ that 
$X_i^+\cap X_j^+ = \emptyset$ when $i\neq j$, since the top slice of a complex in $X_i^+$ is a sum of shifts of $P_i$ and the top slice of a complex in $X_j^+$ is a sum of shifts of $P_j$.  Similarly, $X_i^-\cap X_j^- = \emptyset$ when $i\neq j$.  

We now show that $X_i^+ \cap X_j^- = \emptyset$.  To see this, note that by definition complexes in $X_k^-$ have a shift of $P_k$ as a direct summand in the bottom $\widetilde{o}$-baric slice.  Thus it suffices to note that if $Y\in X_i^+$, then the bottom slice
$Y(\phi_-)$ does not have any direct summands isomorphic to a shift of $P_k$ for any $k=1,\dots,n$.  
For suppose that $Y(\phi_-)$ has a direct summand $N \cong P_k \la \phi_-(Y) \ra$, which for simplicity we assume is in homological degree 0.  Since
$$\Hom(P_k\la\phi_-(Y)\ra,\P_i\la \phi_-(Y)\ra) \cong \Hom(P_k,\P_i) \neq 0,$$ 
it follows that $\Hom(Y,\P_i\la\phi_-(Y)\ra) \neq 0$, too.  To see this, let $0\neq f\in \Hom(Y,\P_i\la\phi_-(Y)\ra)$ be the map on $N$ is given by a $\widetilde{o}$-degree 0 edge of the quiver, extended by 0 to all of $Y$. (The restriction of $f$ to $N$ is a scalar multiple of either an $x^*$, a $y$, or an $e_i$, depending on whether $i<k$, $i>k$, or $i=k$.) Let $f^* \in \Hom(P_i\la\phi_-(Y)-1\ra, Y)$ denote the dual morphism (either an $x$, a $y^*$, or a $z$) from $P_i\la \phi_-(Y)\ra$ into the summand $N$.  Since 
$$
	f\circ f^* = z\in  \Hom(P_i\la \phi_-(Y) - 1 \ra ,\P_i \la \phi_-(Y)\ra )
$$ 
is nonzero, it follows that $f\neq 0$.  But this implies that $Y\notin X_i^+$.  Thus we have shown that the sets $$X_1^-,\dots,X_n^-,X_1^+,\dots,X_n^+$$ are pairwise disjoint.  This completes the proof of $(1)$ and $(2)$.

We now show statement $(5)$:
$$
	Y\in X_j^+ \implies \Sig_i(Y)\in X_i^+.
$$

To see that $\Sig_i(Y)(\phi_+)$ is isomorphic to a sum of shifts of $P_i$, it suffices to note that $$\phi_+(\Sig_i(Y))>\phi_+(Y),$$ 
as the only complex that can appear as an $\widetilde{o}$-baric slice of $\Sig_i(Y)$ in a level greater than $\phi_+(Y)$ is a direct sum of
shifts of $P_i$.  To show $\phi_+(\Sig_i(Y))>\phi_+(Y)$, note that by assumption $Y(\phi_+)$ is isomorphic to a direct sum of shifts of $P_j$.  Let 
$M_j$ be such a summand, which for simplicity we assume lives in homological degree 0.  Now
$$
	\Hom(P_i\la \phi_+(Y) \ra , P_j \la \phi_+(Y) \ra ) \cong \Hom(P_i,P_j)
$$
is clearly non-zero, and if $0 \neq f\in \Hom(P_i\la \phi_+(Y) \ra , P_j \la \phi_+(Y) \ra )$, the map $f' : P_i\la \phi_+(Y) \ra \rightarrow Y$ which maps into $M_j$ via $f$ cannot be homotopic to 0.  It follows that
$$
	\Hom(P_i\la \phi_+(Y)\ra, Y) \neq 0,
$$
and by adjunction
$$
	\Hom(P_i \la \phi_+(Y)+1\ra[-1],\Sig_i(Y)) \cong \Hom(P_i\la \phi_+(Y)\ra, Y) \neq 0,
$$
too.  This shows that $\phi_+(\Sig_i(Y))>\phi_+(Y)$, so that the top level of $\Sig_i(Y)$ is isomorphic to a direct sum of shifts of $P_i$.

We must now show that $\Sig_i(Y)$ satisfies the second requirement for membership in $X_i^+$, namely, that 
$$\Hom(\Sig_i(Y),\P_k\la\phi_-(\Sig_i(Y))\ra = 0 \iff k=i.$$  To this end, note that 
$\Hom(Y,\P_j\la\phi_-(Y)\ra) = 0$ by assumption.  Thus the bottom slice $Y(\phi_-)$ does not have a summand of the form $P_k\la\phi_-(Y)\ra$ for any $k$.  In particular, we have $\phi_-(\Sig_i(Y)) = \phi_-(Y)$, since these two numbers would be different precisely when the bottom slice of $Y$ is a direct sum of shifts of $P_i$.  Since $\phi_-(\Sig_i(Y)) = \phi_-(Y)$, we must show that
$$\Hom(\Sig_i(Y),\P_k\la\phi_-(Y)\ra = 0 \iff k=i.$$  
For this we check the cases $k=i$ and $k\neq i$ separately.  
\begin{itemize}
\item For $k=i$, we have
$$	
\Hom(\Sig_i(Y),\P_i\la\phi_-(Y)\ra = \Hom(Y,\Sig_i^{-1}\P_i\la\phi_-(Y)\ra) = \Hom(Y,\P_i\la\phi_-(Y) - 1\ra) = 0,
$$
since there are no nonzero maps from $Y$ to any complex in $\D_{<\phi_-(Y)}$.  

\item Now suppose $k\neq i$.  We need to show that $\Hom(\Sig_i(Y),\P_k\la\phi_-(Y)\ra) \neq 0$.  We have
$$
\Hom(\Sig_i(Y),\P_k\la\phi_-(Y)\ra \cong \Hom(Y,\Sig_i^{-1}(\P_k)\la\phi_-(Y)\ra)
\cong \Hom(Y,(\P_i \la -1 \ra \oplus \P_i \rightarrow \P_k)\la\phi_-(Y)\ra).
$$
There are no non-zero $\A$-module maps -- much less chain maps -- from the minimal complex of $Y$ to the lowest term $\P_i\la-1+\phi_-(Y)\ra$ above; from this we see that
$$
	\Hom(\Sig_i(Y),\P_k\la\phi_-(Y)\ra\cong \Hom(Y,(\P_i\rightarrow \P_k)\la\phi_-(Y)\ra),
$$
so we must show that the right hand side is non-zero.  We do this by finding an explicit non-zero element in this hom space.
We consider the the minimal complex of the bottom slice $Y(\phi_-)$ of Y:
$$
	Y(\phi_-) = \big( \dots Y_{s-1} \rightarrow Y_s \rightarrow 0 \rightarrow \dots \big),
$$
so that $Y_s$ is the rightmost non-zero term in the minimal complex of the bottom slice.

Suppose first that the chain group $Y_s$ has a summand $M_l\cong P_l\la \phi_-(Y) \ra$ with $l\neq i$, and write $Y_s \cong M_l\oplus Y_s'$.  Let $f$ be a degree 0 morphism (either an $x^*$ or $y$) from $M_l$ to $P_i\la \phi_-(Y) \ra$.  We extend $f$ by 0 to a map on chain groups:
$$\begin{tikzpicture}
  \matrix (m) [matrix of math nodes, row sep=3em, column sep=3em]
    { Y_{s-2} & Y_{s-1}  & M_l \oplus Y_s' & 0 & 0 \\
      0 & 0 & P_i\la \phi_-(Y) \ra & P_k \la \phi_-(Y) \ra & 0  \\ };
  { [start chain] \chainin (m-1-1);
    \chainin (m-1-2);
    { [start branch=A] \chainin (m-2-2)
        [join={node[right,labeled] {0}}];}
    \chainin (m-1-3) [join={node[above,labeled] {\partial_{s-1}}}];
    
    { [start branch=B] \chainin (m-2-3)
        [join={node[right,labeled] {f\oplus 0}}];}
        
          { [start branch=B] \chainin (m-2-2)
        [join={node[right,labeled] {h_1 }}];}

    \chainin (m-1-4) [join={node[above,labeled] {\partial_{s}}}];
    { [start branch=C] \chainin (m-2-4)
        [join={node[right,labeled] {0}}];}
        
           { [start branch=C] \chainin (m-2-3)
        [join={node[right,labeled] {h_2 }}];}
        
    \chainin (m-1-5); }
  { [start chain] \chainin (m-2-1);
    \chainin (m-2-2);
    \chainin (m-2-3) [join={node[above,labeled] {0}}];
    \chainin (m-2-4) [join={node[above,labeled] {d}}];
    \chainin (m-2-5); }
\end{tikzpicture}
$$
Since the differential $d$ above is of $\widetilde{o}$-baric degree 0 and is an $x^*$ or a $y$, the relations in the algebra $\A$ imply that the composition of $d\circ f$ is zero; by the same reasoning, the composition $f\circ \partial_{s-1} = 0$.  Thus the map $f$ may be regarded a well-defined element of $\Hom(Y,(\P_i\rightarrow \P_k)\la\phi_-(Y)\ra)$.  But, as $Y_s$ is the rightmost term in the bottom slice of the minimal complex, there is no possible homotopy $h = (h_1,h_2)$ between $f$ and the 0 map.  Thus in this case $\Hom(Y,(\P_i\rightarrow \P_k)\la\phi_-(Y)\ra) \neq 0$.

On the other hand, suppose that $Y_s \cong P_i^{\oplus m}\la \phi_-(Y) \ra$.  We choose a single summand from $Y_s$ and denote it by $N_i$, so that $Y_s \cong N_i\oplus Y_s'$.  Let $g$ be the degree 0 map from $N_i$ to $P_k\la \phi_-(Y)\ra$ (so $g$ is either an $x^*$ or a $y$).  This $g$ extends to a chain map from $Y$ to $(\P_i\rightarrow \P_k)\la\phi_-(Y)\ra$, which we claim cannot be homotopic to 0.  
$$\begin{tikzpicture}
  \matrix (m) [matrix of math nodes, row sep=3em, column sep=3em]
    { Y_{s-2} & Y_{s-1}  & N_i \oplus Y_s' & 0 & 0 \\
      0 & P_i\la \phi_-(Y) \ra & P_k \la \phi_-(Y) \ra & 0 & 0  \\ };
  { [start chain] \chainin (m-1-1);
    \chainin (m-1-2);
    { [start branch=A] \chainin (m-2-2)
        [join={node[right,labeled] {0}}];}
    \chainin (m-1-3) [join={node[above,labeled] {\partial_{s-1}}}];
    
    { [start branch=B] \chainin (m-2-3)
        [join={node[right,labeled] {g\oplus 0}}];}
        
          { [start branch=B] \chainin (m-2-2)
        [join={node[right,labeled] {h = h'\oplus h''}}];}

    \chainin (m-1-4) [join={node[above,labeled] {\partial_{s}}}];
    { [start branch=C] \chainin (m-2-4)
        [join={node[right,labeled] {0}}];}

    \chainin (m-1-5); }
  { [start chain] \chainin (m-2-1);
    \chainin (m-2-2);
    \chainin (m-2-3) [join={node[above,labeled] {d}}];
    \chainin (m-2-4) [join={node[above,labeled] {0}}];
    \chainin (m-2-5); }
\end{tikzpicture}
$$
To see this, note first that the only possible components to use for a homotopy 
$h = (h',h'')$ must be scalar multiples of the $\widetilde{o}$-degree 0 idempotent $e_i$.
Now, $Y_s'$ is isomorphic to a direct sum of $P_i\la \phi_-(Y)\ra$'s, $h''$ is a map of $\A$ modules whose matrix entries are scalar multiples of the idempotent $e_i$, and the differential $d$ satisfies $$d\circ e_i = d \neq 0.$$  From this we see that $d\circ h'' = 0 \implies h'' = 0$.
But once $h'' = 0$, the homotopy relations imply that we must have $h' \circ \partial_{s-1}=0$, too.  Thus the composition of $\partial_{s-1}$ with the projection $Y_s \rightarrow N_i$ must be 0, and $N_i$ is a direct summand of the bottom slice $Y(\phi(-))$ of $Y$ in the homotopy category.  But, as we have observed earlier, if $Y\in X_j^+$, then the bottom slice of $Y$ cannot have any summands isomorphic to $P_k\la \phi_-(Y) \ra$ for any $k$.  Thus we conclude that the homotopy $h$ cannot exist, and therefore that $$\Hom(Y,(\P_i\rightarrow \P_k)\la\phi_-(Y)\ra) \neq 0.$$

\end{itemize}
This concludes the check of $(5)$.

We now show $(3)$, namely, that $Y\in X_j^-$ and $j\neq i$ implies $\Sig_i(Y)\in X_i^+$.
As before, in order to show that the top slice of $\Sig_i(Y)$ is a direct sum of shifts of $P_i$, it suffices to note that 
$\phi_+(\Sig_i(Y))>\phi_+(Y)$.  Since
$$
	\Hom(\P_i \la\phi_+(Y)+1\ra,\Sig_i(Y)) \cong \Hom(\P_i\la\phi_+(Y)\ra, Y) \neq 0
$$
by assumption, we must have $\phi_+(\Sig_i(Y))>\phi_+(Y)$.
What remains is to show that
$$\Hom(\Sig_i(Y), \P_k\la\phi_-(\Sig_i(Y)\ra) = 0  \iff k= i.$$
But this is clear: by assumption $Y(\phi_-)$ is isomorphic to a direct sum of shifts of $P_j$, and it follows immediately (e.g. from the example computations in Section \ref{sec:functors}) that
$\Sig_i(Y)(\phi_-)$ is isomorphic to a direct sum of shifts of the complex $(P_j\rightarrow P_i)$.  Since maps from $\Sig_i(Y)$ to objects of $D_{\leq \phi_-(\Sig_i(Y))}$ are determined completely by maps out of the bottom slice $\Sig_i(Y)(\phi_-)$, it follows that
$\Hom(\Sig_i(Y), \P_k\la\phi_-(\Sig_i(Y)\ra)$
is isomorphic to a direct sum of hom spaces of the form
$$
	\Hom((P_j\rightarrow P_i)\la\phi_-(\Sig_i(Y)\ra, \P_k\la\phi_-(\Sig_i(Y)\ra) \cong \Hom((P_j\rightarrow P_i), \P_k)
$$
But  now one can see directly using the relations in the algebra $\A$ that the hom space on the right is zero if and only if $k=i$.
The remaining claims $(4)$ and $(6)$ in the proposition now follow by symmetric arguments.
\end{proof}

As a corollary, we obtain Theorem \ref{thm:main} from the introduction.
\begin{corollary}\label{cor:faithful}
The group homomorphism
$$
	\psi: F_n \longrightarrow [\mbox{Aut}(\D)]
$$
from the free group to the isomorphism classes of auto-equivalences of $\D$, is injective.
\end{corollary}

We have proven that $F_n$ acts faithfully on the homotopy category of {\it graded} projective $\A$ modules, where the grading is the one induced by the $\widetilde{o}$ orientation of $Q$.   Of course, one may ask about the faithfulness of the $F_n$ action on the homotopy category of graded modules for one of our other gradings, or about the faithfulness of the action on the homotopy category of ungraded projective modules.  In fact Corollary \ref{cor:faithful} implies that $F_n$ acts faithfully on all of these different triangulated categories.  The main point is that while not all complexes of ungraded $(\A,\A)$ bimodules have graded lifts, the minimal complexes of ungraded $(\A,\A)$ bimodules of the form $\Psi_g$ do all have canonical graded lifts, at least for the gradings we are considering.  From this, one can show directly that if $\Psi_g$ is homotopic to $\A$ as a complex of ungraded bimodules, then there must actually exist an graded homotopy between $\Psi_g$ and $\A$.  Thus faithfulness on the homotopy category of $\widetilde{o}$-graded projective modules implies faithfulness on the homotopy category of ungraded projective modules (and therefore on the homotopy category of graded projective modules for any of our other gradings.)

\section{The Bessis monoid, the $\vec{o}$-grading on $\A$, and dual ping pong}\label{sec:dualpingpong}
In this section we describe the relationship between the Bessis monoid $F_n^+$ of $F_n$ and the triangulated category $\D$.
The basic ingredients are the 2-representation $\Psi$ and the $\vec{o}$-grading on the zigzag algebra $\A$.  

To begin, we explain how the Hurwitz action of the Artin braid group $B_n$ appears naturally from the point of view of the 2-representation $\Psi$.  Our description is very similar to the well-known appearance of $B_n$ in the theory of exceptional sequences of representations of quivers \cite{Crawley-Boevey}.  Actions of $B_n$ on spherical collections in triangulated categories have also been considered before, see for example \cite{BridgelandCY} and references therein.  Our main goal here is to directly relate these kinds of actions to the Hurwitz action of $B_n$ on reduced reflection expressions of $\gamma = \s_1\dots \s_n$.

In this section the grading shift $\la \bullet \ra$ will denote a shift in the $\vec{o}$-grading coming from the orientation $\vec{o}$ of  $Q$, and $\D$ will denote the homotopy category of bounded complexes of finitely-generated $\vec{o}$-graded projective modules.  Given $Y\in \D$, $Y(k)$ will denote the $\vec{o}$-baric slice of $Y$.

\subsection{The Hurwitz action of the braid group on spherical collections}\label{sec:Hurwitzspherical}
One special feature of complexes in $\D$ of the form $\Psi_g(P_i)$ is that their endomorphism algebras 
are isomorphic to $\k[z]/z^2$, the homology ring of the sphere.  For this reason, we refer to these objects as {\it spherical}.
 \footnote{The terms {\it spherical object} and {\it spherical twist} are due to Seidel-Thomas \cite{SeidelThomas}, who define a spherical object to be an object $E$ such that $\oplus_k\Hom(E,E[k])$ is isomorphic to the homology ring of a sphere. The notion of spherical object we use here is quadratic dual to theirs. In particular, if we choose to work instead with the corresponding simple modules over the quadratic dual of $\A$, these simple modules will be spherical in the original sense of Seidel-Thomas.}If $E \cong \Psi_g(P_k)$ is spherical, we denote by $\Phi_{E} \cong \Psi_g \Sig_k \Psi_g^{-1}$ the associated {\it spherical twist,} which is an auto-equivalence of $\D$.

A {\it spherical collection} in $\D$ is a collection 
$$
	\{E_1,E_2,\dots,E_n\}
$$
of finitely-many spherical objects $E_i$.  The auto-equivalence group of $\D$ acts on the set of spherical collections, with an auto-equivalence $\Phi$ acting by
$$
	\Phi\cdot\{E_1,\dots,E_n\} = \{\Phi(E_1),\dots,\Phi(E_n)\}.
$$
It is sometimes convenient to fix an order $(E_1,E_2,\dots,E_n)$ on the objects in a spherical collection; after doing so the Artin braid group $B_n$ then acts on the set of spherical collections, where the action of the generator $\tau_i\in B_n$ is defined by
$$
	\tau_i (E_1,\dots,E_i,E_{i+1},\dots,E_n) = (E_1,\dots, \Phi_{E_{i}}(E_{i+1}),E_i\dots,E_n).
$$
(See, for example, \cite{BridgelandCY}).  We denote the set of ordered spherical collections in $\D$ by
$\cal O_\D$, and, in anticipation of Lemma \ref{lem:o-Hurwitz2}, we refer to this $B_n$ action as the {\it Hurwitz action of the braid group on $\cal O_\D$}.  An interesting question is to describe the orbits of $B_n$ on ordered spherical collections in $\D$; in particular, it is desirable to find orbits where the $B_n$ action is free.
It is also desirable to find conditions on the morphism spaces between objects in a collection which are preserved by the Hurwitz action.  This motivates the following definition.

\begin{definition}
An {\bf $\vec{o}$-spherical collection} in $\D$ is an $n$-tuple
$$
	(E_1,E_2,\dots,E_n)
$$
such that 
\begin{itemize}
\item each $E_i$ is spherical;
\item each $E_i\in \D_0$, where $\D_0$ denotes the heart of the $\vec{o}$-baric structure on $\D$ (see Section \ref{sec:slice});
\item for $i<j$ and $l\in \Z$,
$$
	\Hom(E_i,E_j\la k\ra[l])\neq 0 \implies k = 1;
$$
\item for $i>j$ and $l\in \Z$,
$$
	\Hom(E_i,E_j\la k\ra[l])\neq 0 \implies k = 0;
$$
\end{itemize}
\end{definition}
The basic example of an $\vec{o}$-spherical collection is 
$$
	(P_1,P_2,\dots, P_n).
$$

The proof of the following lemma is almost immediate from the definitions.
\begin{lemma}\label{lem:o-Hurwitz1}
The Hurwitz action of $B_n$ preserves the set of $\vec{o}$-spherical collections.
\end{lemma}

Let $\mathcal{O}^{\vec{o}}_\D(\Psi_\gamma)$ denote the set of $\vec{o}$-spherical collections $(E_1,\dots,E_n)$ such that the composition of their spherical twists is given by the action of the Coxeter element
$$
	\Phi_{E_1}\dots \Phi_{E_n} \cong \Psi_\gamma  \in \mbox{Aut}(\D).
$$  
Recall that $Red(\gamma)$ denotes the set of reduced $\gamma$-reflection expressions for $\gamma\in F_n$.  
A comparison of the definition of the Hurwitz action on $\vec{o}$-spherical collections to that of the Hurwitz action on reduced $\gamma$-reflection expressions for $\gamma$ yields the following.
\begin{lemma}\label{lem:o-Hurwitz2}
The assignment $(\s_1,\dots,\s_n)\mapsto (P_1,\dots,P_n)$ extends uniquely to a map
$$
	f: Red(\gamma)\longrightarrow \mathcal{O}^{\vec{o}}_\D(\Psi_\gamma)
$$
which intertwines the $B_n$ actions.  If $f: (t_1,\dots,t_n) \mapsto (E_1,\dots,E_n)$, then for all $k=1,\dots,n$ the auto-equivalence $\Psi_{t_k}$ is isomorphic to the spherical twist $\Phi_{E_k}$.
\end{lemma}

\begin{proposition}\label{prop:o-Hurwitz}
The Hurwitz action of the braid group $B_n$ on the orbit of the spherical collection $(P_1,\dots,P_n)$ is free.  Thus the assignment
$(\s_1,\dots,\s_n)\mapsto (P_1,\dots,P_n)$ extends uniquely to an isomorphism of sets with a $B_n$ action.
\end{proposition}
\begin{proof}
Let $1\neq \b \in B_n$, and suppose that $\b(P_1,\dots,P_n) = (E_1,\dots,E_n)$.  We must show that for some $i$, 
$E_i$ is not isomorphic to $P_i$.
To see this, consider the Hurwitz action of $B_n$ on the reduced $\gamma$-reflection expressions for $\gamma$, and let
$$
	\b (\s_1,\dots,\s_n) = (t_1,\dots,t_n).
$$
In particular, in the 2-representation $\Psi$, the reflection $t_i$ acts by the spherical twist $\Phi_{E_i}$.
By Artin's Theorem \ref{thm:Artin}
$$
(t_1,\dots,t_n) = (\s_1,\dots,\s_n) \iff \b = 1.
$$
Thus, if $\b\neq 1$, at least one of the reflections $t_i$ is not equal to $\sigma_i$.  Now, by Theorem \ref{thm:main}, it follows that the spherical twists $\Phi_{E_i}$ and $\Phi_{P_i} = \Sigma_i$ are not isomorphic, and thus $E_i$ is not isomorphic to $P_i$.
\end{proof}

\subsection{Reflection complexes}\label{sec:refcomplex}
Let $t\in F_n^+$ be a $\gamma$-reflection. To $t$ we may associate a complex $\cal C_t$ whose associated twist is $t$.  We call $\cal C_t$ the {\it reflection complex} associated to $t$, and in this section we study the basic properties of reflection complexes.  The most important point is that there is a close connection between the morphism spaces between reflection complexes and the topology of the punctured disc.  The basic features of this connection will be summarized in Proposition \ref{prop:equiv} below.

Let $\b\in B_n$ be a braid such that
$$
	\b \cdot (\sigma_1,\dots, \sigma_n) = (t,x_2,\dots,x_n),
$$
where $B_n$ acts on $n$-tuples of $\gamma$-reflections via the Hurwitz action.
Similarly, there exists a complex $\cal C_t$ with
$$
	\b\cdot (P_1,\dots,P_n) = (\cal C_t, \cal C_{x_2},\dots,\cal C_{x_n}).
$$
The complex $\cal C_t$ is well defined independent of the choice of $\b\in B_n$ used to obtain an $n$-tuple of the form $ (t,x_2,\dots,x_n)$.  The complex $\cal C_t$, which we call the {\it reflection complex} associated to $t$, lives in the heart $\D_0$ of the $\vec{o}$-baric structure on $\D$.

As an example, let $t$ be the $\gamma$-reflection $\s_1\s_2\s_1^{-1}$.  The associated reflection complex is
$$
	\cal C_{\s_1\s_2\s_1^{-1}} \cong \Sig_1(P_2) \cong \big(P_2 \xrightarrow{x^*\oplus y^*} P_1^{\oplus 2}).
$$

An alternative description of the reflection complex $\cal C_t$ is as follows: any $\gamma$-reflection $t$ may be written as $t= g \s_i g^{-1}$ for some $g\in F_n$ and some $i=1,\dots,n$, with $g$ determined uniquely up to multiplication on the right by a power of $\s_i$.  The complex $\Psi_g(P_i)$ will agree up to a grading and homological shift with the reflection complex $\cal C_t$ defined above, where the ``up to a shift" is a result of the ambiguity in the choice of $g$.  The advantage of our first definition above is that it uses the Hurwitz action to fix this choice of $g$ uniformly for all $\gamma$-reflections.  However, for some statements it seems more natural to allow for an ``up to shift" ambiguity in the definition of $\cal C_t$.  In particular, for $t$ a $\gamma$-reflection, we define a complex $\mathfrak{C}_t$ by
$$
	\mathfrak{C}_t = \bigoplus_{k\in \Z} \cal C_t [k].
$$
Thus $\mathfrak{C}_t$ is the direct sum of all homological shifts of $\cal C_t$ which live in the $\vec{o}$-baric heart $\D_0$.  Note that changing the definition of $\cal C_t$ by homological shift does not change the object $\mathfrak{C}_t$.  The complexes $\mathfrak{C}_t$ are not objects of $\D$, though the summands of $\mathfrak{C}_t$ are. In particular, for any $Y\in \D$, the space $\Hom(Y,\mathfrak{C}_t)$ is a well-defined finite dimensional $\k$-vector space.

In the statement of the Proposition \ref{prop:equiv} below, we find it convenient to introduce the notation $\doteq$ by writing, for complexes $Y$ and $B$,
$$
	Y \doteq B \iff \text{ there exists an integer } m \text{ such that } Y\cong B[m]. 
$$
When we have fixed a direct sum decomposition $B \cong B_1\oplus \dots \oplus B_k$, we will also write 
$$
	Y \doteq B_1\oplus \dots \oplus B_k \iff \text{ there exist integers } m_1,\dots,m_k \text{ such that }
	Y \cong B_1[m_1] \oplus \dots \oplus B_k[m_k].
$$
In the special case $\cal C_{\s_i} \cong P_i$, we follow the notation of Section \ref{sec:pingpong} and write $ \mathfrak{P}_i:=\mathfrak{C}_{\s_i}$ (though we remind the reader again that in this section the module $P_i$ is endowed with its $\vec{o}$-grading, while in Section \ref{sec:pingpong} the notations $P_i$ and $\P_i$ were used to denote $\widetilde{o}$-graded objects).

\begin{proposition}\label{prop:equiv}
Let $t$ and $u$ be $\gamma$-reflections.  The following are equivalent:
\begin{enumerate}
\item $tu\in \Be^+$ is a simple Bessis element;
\item $tut^{-1}$ is a $\gamma$-reflection;
\item $u^{-1}t u$ is a $\gamma$-reflection;
\item The connecting paths $f_t$ and $f_u$ have isotopic representatives in the unit disc which intersect only at the base point $x_0$, and such that in a sufficiently small neighbourhood of $x_0$, the curve $f_t$ lies below $f_u$.
\item $\Hom(\cal C_u, \mathfrak{C}_t \la k\ra) = 0$ for $k\neq 0$.
\item $\Hom(\cal{C}_t, \mathfrak{C}_u \la k\ra) = 0$ for $k\neq 1$.
\item $\Psi_t(\cal{C}_u)\cong \cal{C}_{tut^{-1}}\in \D_0$;
\item $\Psi_u^{-1}(\cal{C}_t)\cong \cal{C}_{u^{-1}t u}\in \D_0$;
\item $\Psi_u(\cal{C}_t)\in \D_{[0,1]}$ with $\vec{o}$-baric slices
$$
	\Psi_u(\cal{C}_t)(\phi_+) \doteq \oplus_{i=1}^r \cal{C}_u\la 1 \ra \text{ and } \Psi_u(\cal{C}_t)(\phi_-) \doteq \cal{C}_t;
$$
\item $\Psi_t^{-1}(\cal{C}_u)\in \D_{[-1,0]}$ with $\vec{o}$-baric slices
$$
	\Psi_t^{-1}(\cal{C}_u)(\phi_-)\doteq \oplus_{i=1}^r \cal{C}_t\la -1 \ra \text{ and } \Psi_t^{-1}(\cal{C}_u)(\phi_+) \doteq \cal{C}_u;
$$
\end{enumerate}
\end{proposition}
\begin{proof}
We only give a sketch of the proof here (see also \cite{Licata}).  The equivalence of the first four claims is immediate from Bessis' work \cite{Bessis}, and does not involve the 2-representation $\Psi$.  The equivalence of the remaining statements with each other and with items $(1),(2),(3),(4)$ can be shown using the Hurwitz action on the set of $\vec{o}$-spherical collections.

For example, we show the equivalence of $(1)$ and $(5)$.  
Recall that if $(\cal C_{t_1},\dots,\cal C_{t_n})\in \mathcal{O}^{\vec{o}}_\D(\Psi_\gamma)$ is a $\vec{o}$-spherical collection of reflection complexes corresponding to a minimal length reflection expression $$\gamma = t_1t_2\dots t_n,$$ then it follows from the definition of a $\vec{o}$-spherical collection that
for $i<j$
$$
	\Hom(\cal{C}_{t_i},\mathfrak{C}_{t_j}\la k\ra)\neq 0 \implies k = 1,
$$
while for for $i>j$
$$
	\Hom(\cal{C}_{t_i},\mathfrak{C}_{t_j}\la k\ra)\neq 0 \implies k = 0.
$$
Now suppose that $tu$ is a simple Bessis element, and $(t,u,x_1\dots x_{n-2})$ be a minimal length $\gamma$-reflection expression for the Coxeter element $\gamma$.  Then there is some braid $\b\in B_n$ such that 
$$
	\b (P_1,P_2,\dots,P_n) \cong (\cal C_t, \cal C_u, \dots,\cal C_{x_{n-2}}).
$$
Since the $B_n$ action preserves the set of $\vec{o}$-spherical collections, it follows that
$$
	\Hom(\cal{C}_u, \mathfrak{C}_t \la k\ra) = 0 \text{ for } k\neq 0,
$$
giving $(5)$.  

To prove the converse statement $(5) \implies (1)$, for concreteness we will assume that $t = \s_1$, so that $\cal C_t = P_1$ and $\mathfrak{C}_t = \mathfrak{P}_1$.  (The case of more general $t$ can be reduced to this one using the Hurwitz action, although one can also give a more direct proof along the lines indicated below for the case $\cal C_t = P_1$.)  So we now suppose that $\Hom(\cal{C}_u, \mathfrak{P}_1 \la k\ra) = 0$ for $ k\neq 0$.  We claim that $\cal C_u$ is in the full triangulated subcategory generated by $\{P_2,\dots,P_n\}$; that is, we claim that no shifts of $P_1$ appear in the chain groups of the minimal complex of $\cal C_u$.  

For suppose that a shift of $P_1$ appears somewhere in the minimal complex of $\cal C_u$; let us denote such a summand of the chain group of the minimal complex of $\cal C_u$ by $N_1$, which for convenience we assume appears in homological degree 0.  The boundary map of the minimal complex of $\cal{C}_u$ does not involve any non-zero scalar multiples of the map $e_1 : P_1 \rightarrow P_1$, as any appearance of $e_1$  would result in a contractible summand of $\cal{C}_u$.  But the the only $\vec{o}$-degree 0 maps out of $N_1$ are scalar multiples of $e_1$, and all components of the differential in the minimal complex of $\cal C_u$ are of $\vec{o}$-degree 0.  From this it follows that $N_1$ must be in the kernel of the boundary map of the minimal complex of $\cal C_u$.  Now, let $$\iota: P_1 \rightarrow \cal{C}_u$$ be the map which sends $P_1$ to $N_1$ isomorphically via $e_1$.  Since $N_1$ is in the kernel of the boundary map of the minimal complex of $\cal C_u$, $\iota$ is a chain map and defines an element of $\Hom(P_1,\cal{C}_u)$. Let $$z_{N_1}: \cal{C}_u \rightarrow P_1\la1\ra$$ be the map which restricts to multiplication by the loop $z: N_1 \rightarrow P_1\la 1 \ra$ and which is 0 on all other summands of the chain group of the minimal complex of $\cal {C}_u$.  This, too, is a chain map, and composition 
$$z_{N_1} \circ \iota  = z : P_1 \rightarrow P_1 \la 1 \ra$$
is clearly non-zero in the homotopy category; hence $z_{N_1}$ is non-zero in the homotopy category, too.  Thus $\Hom(\cal{C}_u, \mathfrak{P}_1\la1\ra)\neq 0$, contradicting the assumption of $(5)$.

We conclude that if $\Hom(\cal{C}_u, \mathfrak{P}_1 \la k\ra) = 0$ for all $k\neq 0$, then $\cal {C}_u$ must be in the full triangulated subcategory generated by $P_2,\dots,P_n$.  Now an inductive argument using the Hurwitz action implies that if a reflection complex $\cal C_u$ lies in the full triangulated subcategory generated by $P_2,\dots,P_n$, then in fact the reflection $u$ is a divisor of $\gamma' = \s_2\dots \s_n$.  
Thus $\s_1 u$ divides $\s_1 \gamma' = \gamma$, so that $\s_1 u = tu$ is a simple Bessis element, giving $(1)$.  

The equivalence of $(6)$ with $(1)$ is given analogously.
The equivalence of each of $(7)-(10)$ with $(5),(6)$ is obtained directly using the definition of the auto-equivalences $\Psi_t^{\pm 1}, \Psi_u^{\pm 1}$.
\end{proof}

The transitivity Hurwitz action of $B_n$ on $\vec{o}$-spherical collections in the orbit of $(P_1,\dots,P_n)$ has some further important consequences for morphism spaces between reflection complexes.  For example, we have the following lemma, whose proof we leave as an exercise for the interested reader.  (One such proof is by induction on the length of the braid $\b\in B_2 \cong \Z$ needed to obtain $r$ as a component of $\b(t,u)$ under the Hurwitz action.)

\begin{lemma}\label{lem:homs}
Let $t,u$ be $\gamma$-reflections with $tu\in \Be^+$ a Bessis simple element.  Suppose that 
$$
	\Hom(\mathfrak C_t,Y) = \Hom(\mathfrak C_u,Y) = 0.
$$
Then for any $\gamma$-reflection $r$ dividing $tu$,
$$
	\Hom(\mathfrak C_r,Y) = 0.
$$
Similarly, if
$$
	\Hom(Y,\mathfrak C_t) = \Hom(Y,\mathfrak C_u) = 0,
$$
then for any $\gamma$-reflection $r$ dividing $tu$,
$$
	\Hom(Y,\mathfrak C_r) = 0.
$$
\end{lemma}

\subsection{Dual ping pong}
In this section we define the sets required for dual ping pong, as formulated in Lemma \ref{lem:dual}.  We remind the reader that
in this section $\la \bullet \ra$ denotes a shift in the $\vec{o}$-grading, while $\D_k$, $\D_{\geq k}, \D_{\leq k}$, etc. denote the subcategories of $\D$ defined in Section \ref{sec:slice} using the $\vec{o}$-baric structure. 

For $w\in \Be^+$ a simple Bessis element, set
$$
	X_w = \{Y\in \D_{\geq 0} : \Hom(Y,\mathfrak{C}_t) = 0 \iff t \text{ divides } w\}
$$
We now show that the $X_w$ satisfy the requirements for dual ping pong, as formulated in Lemma \ref{lem:dualpingpong}.

\begin{proposition} \label{prop:dual}
The sets $\{X_w\}_{w \in \Be^+}$ are all non-empty and pairwise disjoint.  For $u$ a $\gamma$-reflection and $w\in \Be^+ $, 
$$
	\Psi_u(X_w) \subset X_{lf(uw)}.
$$
\end{proposition}
\begin{proof}
We first show the sets $X_w$ are all non-empty.  For $w\neq \gamma \in \Be^+$, write $\gamma = ({w^\vee})w$, and let 
$$
	Y_w = \bigoplus_{x \text{ a }\gamma \text{-reflection dividing } w^\vee} \cal C_x.
$$
A basic property of the Bessis monoid $F_n^+$ is that, for $t$ a $\gamma$-reflection, $xt\in \Be^+$ for all $x$ dividing $w^\vee$ if and only if $t$ divides $w$.  Thus it follows from from Proposition \ref{prop:equiv} that
$\Hom(Y_w, \mathfrak{C}_t) = 0$ if and only if $xt\in \Be^+$ if and only if $t$ divides $w$; this shows that $Y_w\in X_w$, so that $X_w\neq \emptyset$.  Since $\D_{>0}\subset X_\gamma$, the set $X_\gamma$ is clearly non-empty, too.

The fact that $X_w \cap X_y = \emptyset$ when $w\neq y$ is clear from the definition, since if $Y \in X_w$, then $w$ is the least common multiple of the set
$\{t \mid \Hom(Y,\mathfrak{C}_t) = 0\}$, and thus $Y$ determines $w$.

We now show that $\Psi_u(X_w) \subset X_{lf(uw)}.$  
Let $a=lf(uw)$, and write $uw = a b = u a' b$, with $a'b=w$.
Let $Y\in X_w$.  We must now show that $\Psi_uY\in X_a$, that is, we must show that
$$
	\Hom(\Psi_uY, \mathfrak{C}_t)=0 \iff t \text{ divides } a.
$$
Suppose first that $t$ divides $a$.  One possibility is that $t=u$, in which case
$$
	\Hom(\Psi_uY,\mathfrak{C}_u) \cong \Hom(Y, \Psi_u^{-1} \mathfrak{C}_u) \cong \Hom(Y,\mathfrak{C}_u\la -1 \ra).
$$
As $Y\in \D_{\geq 0}$, the summands of $\mathfrak{C}_u\la-1\ra $ are in $\D_{<0}$, and 
$$\Hom(\D_{\geq 0},\D_{<0})=0,$$
we have that $\Hom(Y,\mathfrak{C}_u\la -1 \ra) = 0$.

Next, suppose that $t$ divides $a'$.  Since $t$ divides $w$ and $Y\in X_w$, we know that
$$
	\Hom(Y, \mathfrak{C}_t) = 0.
$$
Also, as $ua'\in \Be^+$ is a simple Bessis element and $ut$ divides $ua'$, we know that $ut\in \Be^+$, too.  Now, by Proposition \ref{prop:equiv} (10), we have that
$\Psi_u^{-1} \cal C_t$ is isomorphic to a mapping cone of a direct sum of $\cal C_u\la-1\ra$'s mapping to $\cal C_t$.
Since the $\cal C_u\la-1\ra$ terms of $\Psi_u^{-1} \cal C_t$ live in $\D_{< 0}$, it follows that for $Z\in \D_{\geq 0}$, 
$$
	\text{if } \Hom(Z,\mathfrak{C}_t) = 0 \text{ then } \Hom(Z,\Psi_u^{-1} \mathfrak{C}_t) = 0.
$$
Taking $Z = Y$, we then have 
$$
	\Hom(Y,\mathfrak{C}_t) = 0 \implies  \Hom(\Psi_u Y, \mathfrak{C}_t) \cong \Hom(Y, \Psi_u^{-1}\mathfrak{C}_t) = 0.
$$ 
Thus $\Hom(\Psi_uY, \mathfrak{C}_t) = 0 $ for $t=u$ and also for $t$ dividing $a'$.   Now Lemma \ref{lem:homs} shows that $\Hom(\Psi_uY, \mathfrak{C}_t) = 0 $ for any $t$ dividing $ua'=a = lf(uw)$.

We now show the converse statement:
$$
	\Hom(\Psi_uY,\mathfrak{C}_t) = 0 \implies t \text{ divides }a.
$$
Let $a''$ be the least common multiple of $\{t \mid \Hom(\Psi_uY, \mathfrak{C}_t) = 0\}$.  Since $u$ is in this set, $u$ divides $a''$, and we may write
$
	a'' = u a'''
$
for a simple Bessis element $a'''$.  If we show that $a'''$ divides $w$, then since $u a''' = a''$, it will follow that $a''$ divides $a = lf(uw)$.  

To see that $a'''$ divides $w$, suppose that $t$ is a $\gamma$-reflection dividing $a'''$.  We claim that $\Hom(Y, \mathfrak{C}_t)=0$.
As before, we know that
$$
	\text{ if } \Hom(Y, \Psi_u^{-1}\mathfrak{C}_t)=0 \text{ then } \Hom(Y, \mathfrak{C}_t) = 0.
$$
Since 
$$
	\Hom(\Psi_uY, \mathfrak{C}_t)\cong \Hom(Y, \Psi_u^{-1}\mathfrak{C}_t)=0,
$$
and $Y\in X_w$, we conclude then $t$ divides $w$.  Since $t$ was an arbitrary factor of $a'''$, it follows that $a'''$ divides $w$, and $a''=ua'''$ divides $a=lf(uw)$, as desired.
\end{proof}

\begin{remark}
The statements Proposition \ref{prop:dual} above together with the dual ping pong Lemma \ref{lem:dualpingpong} together imply the faithfulness of the $F_n$ action on $\D$,  so it might appear as though we have given a second independent proof of Theorem \ref{thm:main}.
However, in preparation for Proposition \ref{prop:dual}, we used Theorem \ref{thm:main} to prove that the Hurwitz action of $B_n$ on $\vec{o}$-spherical collections is free. Nevertheless, one can give a proof of Proposition \ref{prop:dual} which does not use Theorem \ref{thm:main} (by, for example, proving the freeness of the Hurwitz action on $\vec{o}$-spherical collections directly) and so in principle one can prove Theorem \ref{thm:main} using either ping pong or dual ping pong.  We give a dual ping pong faithfulness proof in setting of categorical actions of ADE spherical Artin-Tits groups in \cite{LicataQueffelec}.
\end{remark}

\section{Metrics on $F_n$ from homological algebra}\label{sec:metrics}
Let $o$ be an orientation of $Q$, and $\{\D_k\}$ the associated $o$-baric structure of $\D$, as defined in Section \ref{sec:slice}. For any auto-equivalence $\b$ of $\D$ we may define a closed interval 
$$[\phi_-(\b),\phi_+(\b)]\subset \Z$$ 
as the smallest interval such that 
\begin{equation}\label{eq:interval}
\Psi_\b(\D_0) \subset \D_{[\phi_-(\b),\phi_+(\b)]}.
\end{equation}
There is no a priori reason to consider only $o$-baric structure on $\D$ for such a definition, and one could (and should) consider instead $t$-structures, co-$t$ structures,  Bridgeland slicings or co-slicings, etc.  In favourable situations, the length of this interval will give a metric on the auto-equivalence group of $\D$.  It would be interesting to study the resulting metrics on the free group coming from categorical structures on $\D$ more systematically, and we certainly have not done that here.  However, the ping pong and dual ping pong constructions considered in Sections \ref{sec:pingpong} and \ref{sec:dualpingpong} allow us to at least give a combinatorial description of the metrics arising from the $\widetilde{o}$- and $\vec{o}$-baric structures on $\D$.  We collect the relevant statements in two theorems below.  The proofs of these theorems are very similar to the proofs of Propositions \ref{prop:pingpong} and \ref{prop:dual}, and as a result we omit the details.  (In fact, the easiest way to prove these statements is via an examination of the ping pong and dual ping pong games constructed in Propositions \ref{prop:pingpong} and \ref{prop:dual}.  We also refer also \cite{LicataQueffelec},\cite{Licata}, where we prove analogous metric statements for spherical and right-angled Artin-Tits groups.)

For the first statement, we consider the symmetric orientation $\widetilde{o}$, which was used in Section \ref{sec:pingpong} for our construction of ping pong and proof of Theorem \ref{thm:main}.  Thus the integers $\phi_+(\b)$ and $\phi_-(\b)$ are determined by the interval in (\ref{eq:interval}), where $\D_k$ are given by the $\widetilde{o}$-baric structure on $\D$.

\begin{theorem}\label{thm:metric1}
For $n\geq 2$ and all $\b \in F_n$, $\phi_-(\b)\leq 0$ and $\phi_+(\b) \geq 0$.  Furthermore, 
\begin{itemize}
\item $\phi_+(\b)$ is equal to the number of positive generators from the set $\{\s_1,\dots,\s_n\}$ in a reduced expression 
$$
	\b = \sigma_{i_1}^{\epsilon_1} \dots \sigma_{i_r}^{\epsilon_r}, \ \ \epsilon_i\in \pm 1;
$$ 
\item  $-\phi_-(\b)$ is equal to the number of negative generators from the set $\{\s_1^{-1},\dots,\s_n^{-1}\}$ in a reduced expression 
$$
	\b = \sigma_{i_1}^{\epsilon_1} \dots \sigma_{i_r}^{\epsilon_r},  \ \ \epsilon_i\in \pm 1;
$$ 
\item the $\widetilde{o}$-baric spread $l_{\widetilde{o}}(\b) = \phi_+(\b) - \phi_-(\b)$ is equal to the word length of $\beta$ in the standard generators $\{\sigma_i^{\pm1}\}$ of $F_n$.
\end{itemize}
\end{theorem}

We now replace the symmetric orientation $\widetilde{o}$ on $Q$ by the ordered orientation $\vec{o}$ on $Q$; the $\vec{o}$-grading on $\A$ was used in Section \ref{sec:dualpingpong} for dual ping pong.  Thus in the following theorem, the integers 
$\phi_+(\b)$ and $\phi_-(\b)$ are defined by the interval in (\ref{eq:interval}), where now $\D_k$ are given by the $\vec{o}$-baric structure on $\D$.
\begin{theorem}\label{thm:metric2}
For any $\b\in F_n$, we have the following:
\begin{itemize}
\item If $\phi_+(\b) \geq 0$, then $\phi_+(\b)$ is equal to the number of positive simple Bessis elements $w\in \B e^+$ in any minimal length expression 
$$
	\b = w_{i_1}^{\epsilon_1} \dots w_{i_r}^{\epsilon_r},  \ \ \epsilon_i\in \pm 1
$$ 
for $\b$ as a word in the generating set $(\Be^+)^{\pm 1}$.
\item If $\phi_+(\b) \leq 0$, then $\b\in (F_n^+)^{-1}$ is in the negative Bessis monoid.  Moverover, in this case $-\phi_+(\b)$ is equal to the number of $\gamma^{-1}$'s in any minimal length expression 
$$
	\b = w_{i_1}^{\epsilon_1} \dots w_{i_r}^{\epsilon_r},  \ \ \epsilon_i\in \pm 1
$$ 
for $\b$ as a word in the generating set $(\Be^+)^{\pm 1}$.
\item  If $\phi_-(\b) \leq 0$, then $-\phi_-(\b)$ is equal to the number of negative simple Bessis elements $w^{-1}$ in any minimal length expression 
$$
	\b = w_{i_1}^{\epsilon_1} \dots w_{i_r}^{\epsilon_r},  \ \ \epsilon_i\in \pm 1
$$ 
for $\b$ as a word in the generating set $(\Be^+)^{\pm 1}$.
\item If $\phi_-(\b) \geq 0$, then $\b\in F_n^+$ is in the positive Bessis monoid.  Moreover, in this case $\phi_-(\b)$ is equal to the number of $\gamma$'s in any minimal length expression 
$$
	\b = w_{i_1}^{\epsilon_1} \dots w_{i_r}^{\epsilon_r},  \ \ \epsilon_i\in \pm 1
$$ 
for $\b$ as a word in the generating set $(\Be^+)^{\pm 1}$.
\item Let 
$$
	\phi_+^*(\b) = max\{\phi_+(\b),0\} \text{ and } \phi_-^*(\b) = min\{\phi_-(\b),0\},
$$
so that $\phi_+^*(\b) \geq 0$ and $\phi_-^*(\b)\leq 0$.
The $\vec{o}$-baric spread 
$$
	l_{\vec{o}}(\b) = \phi^*_+(\b) - \phi^*_-(\b)
$$ 
is equal to the word length of $\beta$ in the generating set $(\Be^+)^{\pm1}$ of simple Bessis elements and their inverses.
\end{itemize}
\end{theorem}

Thus Theorems \ref{thm:metric1} and \ref{thm:metric2} give homological constructions of the standard and dual word-length metrics on $F_n$.

\subsection{The exotic metric on $F_n$}

We close by defining a distinguished homological metric on the free group arising from the faithful 2-representation $\Psi$.
Recall that the triangulated category $\D$ has a distinguished $t$-structure $\{\D^k\}$ whose heart $\D^0$ is the abelian category of linear complexes of graded projective $\A$-modules, where {\it here the grading on $\A$ we consider is the grading by path-length} (see Section \ref{sec:gradings}).  Just as for the baric structures considered in the previous section, this $t$-structure gives rise to a metric on $F_n$, and a basic problem is to describe it combinatorially.    

In more detail, for $\b\in F_n$, define $[\phi_-(\b),\phi_+(\b)]\subset \Z$ to be the smallest interval such that 
$$
	\Psi_\b(\D^0) \subset \D^{[\phi_-(\b),\phi_+(\b)]}.
$$
It is not difficult to see that for $n\geq 2$ and any $\b\in F_n$, $-\infty < \psi_-(\b) \leq 0$ and $0 \leq \phi_+(\b) < \infty$.
Using the injectivity of the 2-representation $\Psi$, one can also show the following.
\begin{proposition}
The function $d:F_n\times F_n \longrightarrow \Z_{\geq 0}$ defined by
$$
	d_{exotic}(\alpha,\b) = \phi_+(\b^{-1}\alpha) -  \phi_-(\b^{-1}\alpha) 
$$
is a (left-invariant) metric on $F_n$, $n\geq 2$.
\end{proposition}\label{prop:metric3}
The triangle inequality and the symmetry of $d_{exotic}$ in the arguments $\alpha$ and $\beta$ both follow immediately from the definitions; the non-trivial point is to prove that $$\text{ if } d_{exotic}(\alpha,\b) = 0 \text{ then } \alpha = \b.$$  
Thus, to prove Proposition \ref{prop:metric3}, one must show that
\begin{equation}\label{eq:strongfaithfulness}
	\beta(\D^0) \subset \D^0 \iff \beta = 1.
\end{equation}
This is in principle a slightly stronger condition than the faithfulness of the action of $F_n$ on $\D$.  However, the faithfulness of the action of $F_n$ on $\D$ implies that for $\b\neq 1$, the complex of $(\A,\A)$ bimodules defining $\Psi_\b$ is not homotopic to the $(\A,\A)$ bimodule $\A$. Now, by examining the chain groups of the (minimal complexes of) the top and bottom slices of $\Psi_\b$, one can show that there is either a projective module $P_j$ such that the top $t$-slice of $\Psi_\b(P_j)$ is in $\D^{>0}$ or a projective module $P_j$ such that the bottom $t$-slice of $\Psi_\b(P_j)$ is in $\D^{<0}$.  From this it follows that 
$$
	\Psi_\b (\oplus_{j=1}^n P_j) \in \D^0 \iff \b = 1,
$$
which implies statement $(2)$.

We refer to the metric $d_{exotic}$ as the {\it exotic metric} on $F_n$, and pose the following problem.

\begin{problem}\label{prob:exotic}
Give a combinatorial description of the exotic metric on $F_n$.
\end{problem}

If $\beta$ is an element of the standard positive monoid, so that a minimal expression for $\beta$ in the generators $\{\sigma_i^{\pm1}\}$ uses only positive generators $\{\s_i\}$, then one can show that $d(\beta,1)$ is equal to the length of $\beta$ with respect to the generating set of canonical positive lifts from the associated universal Coxeter group 
$$
	W_n  = F_n/\la \s_i^2 \ra.
$$
We denote by $d_{Cox}$ the word-length metric on $F_n$ whose length one elements are these positive lifts and their inverses.
Thus, for $n=2$, the free group elements $\alpha$ with $d_{Cox}(\alpha) = 1$ are
$$
	\s_1,\ \s_2,\ \s_1\s_2,\ \s_2\s_1,\ \s_1\s_2\s_1,\ \s_2\s_1\s_2,\ \s_1\s_2\s_1\s_2,\dots
$$
and their inverses.  In fact, when $n=2$ one can show that for all $\alpha, \beta\in F_2$, 
$$
	d_{exotic}(\alpha,\beta) = d_{Cox}(\alpha,\beta).
$$ 
This  solves Problem \ref{prob:exotic} when $n=2$.  

When $n>2$, the situation appears to be more subtle.  While we always have 
$$
	d_{exotic}(\alpha,\beta) \leq d_{Cox}(\alpha,\beta),$$
typically we will not have equality (though we will have equality when $\beta^{-1}\alpha$ is in the standard positive monoid.)  
For example, in $F_3$, if
$$
\alpha = \s_2\s_1 \text{ and } \b = \s_1\s_3\s_1^{-1},
$$
then $d_{Cox}(\alpha,\beta)= 3$, but $d_{exotic}(\alpha,\beta) = 2$.  In fact we have no reason to expect that $d_{exotic}$ is a word-length metric when $n>2$.

\bibliographystyle{amsplain}
\bibliography{free}

\end{document}